\documentclass[12pt,a4paper]{amsart}
\usepackage{graphicx,latexsym,amsfonts,amsmath,amssymb,rotating,txfonts,mathrsfs,enumerate}
\usepackage[latin1]{inputenc}
\usepackage{epic}
\usepackage{curves}
\usepackage{pdfsync}
\input xy
\xyoption{all}

\newcommand{\Hom}{ \,{\rm Hom} \,}
\newcommand{\Sym}{ \,{\rm Sym} \,}

{\bf}{\rm}
\newtheorem{theorem}{Theorem}[section]
\newtheorem{proposition}[theorem]{Proposition}
\newtheorem{corollary}[theorem]{Corollary}
\newtheorem{lemma}[theorem]{Lemma}
\newtheorem{definition}[theorem]{Definition}
\newtheorem{remark}[theorem]{Remark}

\newtheorem{example}[theorem]{Example}
{\bf}{\it}

\newcommand{\CC}{{\mathbb C }}

\newcommand{\QQ}{{\mathbb Q }}

\newcommand{\ff}{{\mathbf f }}

\newcommand{\calo}{\mathcal{O}}
\newcommand{\calf}{\mathcal{F}}
\newcommand{\cale}{\mathcal{E}}
\newcommand{\cals}{\mathcal{S}}

\newcommand{\calz}{\mathcal{Z}}
\newcommand{\calm}{\mathcal{M}}

\newcommand{\reg}{\mathrm{reg}}

\newcommand{\epd}[1]{\mathrm{eP}[#1]}

\newcommand{\emu}{\mathrm{emult}}
\newcommand{\um}[1]{\mathrm{sum}(#1)}
\newcommand{\dist}{{\mathrm{dst}}}
\newcommand{\lead}{\mathrm{lead}}
\newcommand{\Bipi}{{\boldsymbol{\Pi}_k}}
\newcommand{\bipi}{{\boldsymbol{\pi}}}
\newcommand{\sg}[1]{\mathcal{S}_{#1}}
\newcommand{\res}{\operatornamewithlimits{Res}}

\newcommand{\ires}{\res_{z_1=\infty}\res_{z_{2}=\infty}\dots\res_{z_k=\infty}}
\newcommand{\sires}{\res_{\mathbf{z}=\infty}}
\newcommand{\dbz}{\,d\mathbf{z}}

\newcommand{\coeff}{\mathrm{coeff}}

\newcommand{\symdot}{\mathrm{Sym}^{\le k}\CC^n}
\newcommand{\symdotx}{\mathrm{Sym}^{\le k}T_X^*}

\newcommand{\grass}{\mathrm{Grass}}

\newcommand{\flag}{\mathrm{Flag}}
\newcommand{\TT}{\mathrm{T}}
\newcommand{\Fl}{\mathcal{F}}

\newcommand{\cotx}{T_X^*}

\newcommand{\bz}{\mathbf{z}}

\newcommand{\bv}{\mathbf{v}}

\newcommand{\kt}{{K}}

\newcommand{\OO}{\mathcal{O}}

\newcommand{\jetk}[2]{J_{k}({#1},{#2})}

\newcommand{\jetreg}[2]{J_{k}^{\mathrm{reg}}({#1},{#2})}

\newcommand{\tc}{\hat T}

\newcommand{\liet}{{\mathfrak t}}

\newcommand{\GL}{\mathrm{GL}}
\newcommand{\sym}{\mathrm{Sym}}

\newcommand{\Hilb}{\mathrm{Hilb}}
\newcommand{\CHilb}{\mathrm{CHilb}}
\def\a{\alpha}

\def\g{\gamma}

\def\l{\lambda}

\def\s{\sigma}

\def\vp{\varphi}

\title{Tautological integrals on curvilinear Hilbert schemes} 

\author{Gergely B\'erczi, Oxford}
\address{Mathematical Institute, University of Oxford, Andrew Wiles Building, OX2 6GG Oxford, UK} \email{berczi@maths.ox.ac.uk}
\date{}
\begin{document}

\begin{abstract}
We take a new look at the curvilinear Hilbert scheme of points on a smooth projective variety $X$ as a projective completion of the non-reductive quotient of holomorphic map germs from the complex line into $X$ by polynomial reparametrisations. Using an algebraic model of this quotient coming from global singularity theory we develop an iterated residue formula for tautological integrals over curvilinear Hilbert schemes.
\end{abstract}

\maketitle

\section{Introduction}\label{sec:intro}

Let $X$ be a smooth projective variety of dimension $n$ and let $F$ be a rank $r$ algebraic vector bundle on X. Let $X^{[k]}$ denote  the Hilbert scheme of length $k$ subschemes of $X$ and let $F^{[k]}$ be the corresponding tautological rank $rk$ bundle on $X^{[k]}$ whose fibre at $\xi \in X^{[k]}$ is $H^0(\xi,F|_\xi)$. 

Let $\Hilb^k_0(\CC^n)$ be the punctual Hilbert scheme defined as the closed subset of $(\CC^n)^{[k]}=\Hilb^k(\CC^n)$ parametrising subschemes supported at the origin. Following Rennemo \cite{rennemo} we define punctual geometric subsets as constructible subsets $Q\subseteq \Hilb^k_0(\CC^n)$ which are union of isomorphism classes of schemes, that is, if $ \xi\in Q$ and $\xi'\in \Hilb^k_0(\CC^n)$ are isomorphic (they have isomorphic coordinate rings) then $\xi'\in Q$. Geometric subsets of $X^{[k]}$ of type $(Q_1,\ldots, Q_s)$ are those generated by finite unions, intersections and complements from sets of the form 
\[P(Q_1,\ldots, Q_s)=\{\xi\in X^{[k]}|\xi=\xi_1\sqcup \ldots \sqcup \xi_s, \xi_i \in Q_i\}.\]
For a geometric subset $\calz$ let $\overline{\calz}$ denote its Zariski closure in $X^{[k]}$. Let $M(c_1,\ldots, c_{rk})$ be a monomial in the Chern classes $c_i=c_i(F^{[k]})$ of weighted degree equal to $\dim \overline{\calz}$ where the weight of $c_i$ is $2i$. If $\alpha_M \in \Omega^*(\calz)$ is a closed compactly supported differential form representing the cohomology class of $M(c_1,\ldots, c_{rk})$ then the Chern numbers 
\[[\overline{\calz}]\cap M(c_1,\ldots, c_{rk})=\int_{\overline{\calz}} \alpha_M\]
are called tautological integrals of $F^{[k]}$. Rennemo \cite{rennemo} shows that these integrals can be expressed in terms of the Chern numbers of $X$ and $F$. 
\begin{theorem}[Rennemo \cite{rennemo}]\label{rennemo}
Let $\calm_{r,n}$ denote the set of weighted-degree-$n$ monomials in the Chern classes $c_1(F),\ldots, c_r(F)$ and $c_1(X),\ldots, c_n(X)$. For $S \in \calm_{r,n}$ let $\alpha_S \in \Omega^{top}(X)$ be a closed compactly supported differential form representing the cohomology class of $S$ and let $y_S=\int_X \alpha_S$ denote the corresponding intersection number. Let $\calz \subset X^{[k]}$ be a geometric subset. Then for any Chern monomial $M=M(c_1,\ldots, c_{rk})$ of weighted degree $\dim \overline{\calz}$ there is a polynomial $R_M$ in $|\calm_{r,n}|$ variables depending only on $M$ such that  
\[[\overline{\calz}] \cap M(c_1,\ldots, c_{rk})=R_M(y_S:S\in \calm_{r,n}).\]
\end{theorem}

The proof of \cite{rennemo} is nonconstructive and based on the fact that an element in the cohomology ring of a Grassmannian is a polynomial in the Chern classes of the universal bundle. Lacking a method of obtaining information about this polynomial, there is no apparent way of turning this proof into an algorithm.
Explicit expressions for tautological integrals are not known in general.  On surfaces the method of \cite{egl} yields a recursion which in principle computes the universal polynomial explicitly. The top Segre classes of tautological bundles over surfaces provides an example of this problem and the conjecture of Lehn \cite{lehn} has been recently proved by Marian, Oprea and Pandharipande \cite{mop} for K3 surfaces using virtual localisation. 

Let $X$ be a smooth projective variety of dimension $n$. This paper provides a closed iterated residue formula for tautological integrals over the simplest geometric subsets $P(Q)$ where $s=1$ and the punctual geometric subset $Q$ is defined as
\[Q=\{\xi\in \Hilb^k_0(\CC^n):\calo_\xi \simeq \CC[z]/z^{k}\}.\]
We will see that $\overline{Q}$ is an irreducible component of the punctual Hilbert scheme. 
Points of $P(Q)$ correspond to curvilinear subschemes on $X$, i.e subschemes contained in the germ of some smooth curve on X. In other words, these are the limit points on $X^{[k]}$ where $k$ distinct points come together along a smooth curve. We denote this curvilinear locus by $CX^{[k]}$ and its closure by $\overline{CX}^{[k]}$ which we call the curvilinear Hilbert scheme.  

The main result of the present paper if the following

\begin{theorem}\label{main}
Let $k\ge 1$ and $M(x_1,\ldots, x_{r(k+1)})$ be a monomial of weighted degree $\dim \overline{CX}^{[k+1]}=n+(n-1)k$ in the variables $x_i$ of weight $2i$ for $1\le i \le r(k+1)$. Let $c_i=c_i(F^{[k+1]})$ denote the $i$th Chern class of the tautological rank $r(k+1)$ bundle on $X^{[k+1]}$. Then  
\[[\overline{CX}^{[k+1]}] \cap M(c_1,\ldots, c_{r(k+1)})=\int_X \sires \frac{(-1)^{nk}\prod_{i<j}(z_i-z_j)Q_k(\bz)M(c_i(z_i+\theta_j,\theta_j))d\bz}{\prod_{i+j\le l\le k}(z_i+z_j-z_l)(z_1\ldots z_k)^n}\prod_{i=1}^k s_X\left(\frac{1}{z_i}\right)\]
where $s_X\left(\frac{1}{z_i}\right)=1+\frac{s_1(X)}{z_i}+\frac{s_2(X)}{z_i^2}+\ldots +\frac{s_n(X)}{z_i^n}$
is the total Segre class at $1/z_i$ and the iterated residue is equal to the coefficient of $(z_1\ldots z_k)^{-1}$ in the expansion of the rational expression in the domain $z_1\ll \ldots \ll z_k$. Finally $Q_k(\bz)$ is a homogeneous polynomial invariant of Morin singularities given as the equivariant Poincar\'e dual of a Borel orbit defined below under the explanation. 
\end{theorem}

Explanation and features of the residue formula:
\begin{itemize}
\item The iterated residue gives a degree $n$ symmetric polynomial in Chern roots of $F$ and Segre classes of $X$ reproving Theorem \ref{rennemo} This shows that the dependence on Chern classes of $X$ in fact can be expressed via the Segre classes of $X$.
\item For fixed $k$ the formula gives a universal generating series for the integrals as the dimension increases. 
\item The Chern class $c_i(z_i+\theta_j,\theta_j)$ is the coefficient of $t^i$ in 
\[c(F^{[k+1]})(t)=\prod_{j=1}^r(1+\theta_jt)\prod_{i=1}^k\prod_{j=1}^r(1+z_it+\theta_jt),\] 
that is, the $i$th Chern class of the bundle with formal Chern roots $\theta_j,z_i+\theta_j$. 
\item The quick description of $Q_k$ is the following. The $\GL_k$-module
  of 3-tensors  \\$\mathrm{Hom}(\CC^k,\mathrm{Sym}^2\CC^k)$ has a diagonal decomposition
\[ \mathrm{Hom}(\CC^k,\mathrm{Sym}^2\CC^k)=\bigoplus \CC q^{mr}_l,\;  1\leq m,r,l \leq k,
\]
where the $T_k$-weight of $q^{mr}_l$ is $(z_m+z_r-z_l)$. Let 
\[\epsilon=\sum_{m=1}^k\sum_{r=1}^{k-m}q^{mr}_{m+r} \subset W=\bigoplus_{1\leq m+r\leq l\leq k} \CC q^{mr}_l\subset \mathrm{Hom}(\CC^k,\mathrm{Sym}^2\CC^k).\]
Then $Q_k(\bz)=\epd{\overline{B_k\epsilon},W}$ is the equivariant Poincar\'e dual of the Borel orbit $\overline{B_k\epsilon}$ in $W$, see \S\ref{sec:final} for details.The list of these polynomials begins as follows:
\[Q_{1}=Q_2=Q_3=1,\ Q_4=2z_1+z_2-z_4.\]
In principle, $Q_k$ may be calculated for each concrete $k$ using a
computer algebra program, but at the moment, we do not have an
efficient algorithm for performing such calculations for large $k$ and $Q_k$ is known for $k\le 6$, see \S\ref{sec:final}.
\end{itemize}

The intersection theory of the Hilbert scheme of points on surfaces has been extensively studied and it 
can be approached from different directions. One is the inductive recursions set up in \cite{egl}, an other possibility is using Nakajima calculus \cite{nakajima,lehn}. By these methods, the integration of tautological classes is reduced to a combinatorial problem. Another strategy is to prove an equivariant version of Lehn's conjecture for the Hilbert scheme of points of $\CC^2$ via appropriately weighted sums over partitions. More recently Marian, Oprea and Pandharipande proved a conjecture of Lehn \cite{lehn} on integrals of top Segre classes of tautological bundles over the Hilbert schemes of points over surfaces in the K3 case via virtual localisation on the Quot schemes of the surface.  

In this paper we suggest a new approach by taking a look at Hilbert schemes of points from a different perspective.  We work in arbitrary dimension and not just over surfaces. Of course, for $n\ge 3$ not much is known about the irreducible components and singularities of the punctual Hilbert scheme $\Hilb^k_0(\CC^n)$ so we only focus on the curvilinear component. The crucial observation is that the for $k\ge 1$ the punctual curvilinear locus $CX^{[k+1]}_p$ at $p\in X$ can be described as the non-reductive quotient of $k$-jets of holomorphic map germs $(\CC,0) \to (X,p)$ by polynomial reparametrisations of $\CC$ at the origin. If $u,v$ are positive integers let $J_k(u,v)$ denote the vector space of $k$-jets of holomorphic maps $(\CC^u,0)\to (\CC^v, 0)$ at the origin, that is, the set of equivalence classes of maps $f:(\CC^u,0)\to (\CC^v,0)$ with $f'\neq 0$, where $f\sim g$ if and only if $f^{(j)}(0)=g^{(j)}(0)$ for all $j=1,\ldots, k$. One can compose map-jets via substitution and elimination of terms
of degree greater than $k$; this leads to the composition maps
\begin{equation}\label{composition}
J_k(u,v) \times J_k(v,w) \to J_k(u,w).
\end{equation}
In particular, if $\jetreg uv$ denotes the jets $f=(f',\ldots f^{(k)})$ with $f'\neq 0$ (the regular jets) then \eqref{composition} defines an action of the reparametrisation group $\jetreg 11$ on the regular jets $\jetreg 1n$. The punctual curvilinear locus (as a set) can be identified with the quasi-projective quotient 
\[CX^{[k+1]}_p\simeq \jetreg 1n/\jetreg 11\]
and the curvilinear Hilbert scheme is a fibrewise projective compactification of this non-reductive quotient over $X$ as $p$ moves on $X$.

Using an algebraic model coming form global singularity theory (we call this the test curve model) we reinterpret the natural embedding of the punctual Hilbert scheme $X^{[k+1]}_p=\Hilb^{k+1}_0(\CC^n)$ into the Grassmannian of codimension $k$ subspaces in the maximal ideal $\mathfrak{m}=(x_1,\ldots, x_n)$ as a parametrised map 
\[\phi^\grass:\overline{CX}^{[k+1]}_p \hookrightarrow \grass_k(\symdot) \text{ where } \symdot=\bigoplus_{i=1}^k \sym^i \CC^n.\]
We then apply a two-step equivariant localisation on the fibre $\overline{CX}^{[k+1]}_p$ following the strategy of \cite{bsz}. However, for tautological integrals we need to modify the proof in \cite{bsz} in two crucial points:
\begin{itemize}
\item First, the main obstacle to apply localisation directly is that we don't know which fixed points of the ambient Grassmannian sit in the image $\overline{CX}^{[k+1]}_p$. However, for $k+1\le n$ we prove in \cite{bsz} a residue vanishing theorem which tells that after transforming the localisation formula into an iterated residue only one distinguished fixed point of the torus action contributes to the sum. This mysterious property remains valid for tautological integrals but its proof needs a more detailed study of the rational differential form. 
\item Second, we need to extend the formula to the domain where $k+1>n$, that is, the number of points is larger than the dimension. The trick here is to increase the dimension of the variety and study $\Hilb^{k+1}_0(\CC^n)$ as a subvariety of $\Hilb^{k+1}_0(\CC^{k+1})$. 
 \end{itemize}
 
 The developed method reflects a surprising feature of curvilinear Hilbert schemes: in order to evaluate tautological integrals and make the residue vanishing principle work we need to increase the dimension of the variety first and work in the range where the number of points does not exceed the dimension. 
 
\noindent\textbf{Acknowledgments} I warmly thank Frances Kirwan and J$\o$rgen Vold Rennemo for the valuable discussions. This paper has outgrown from \cite{bsz} and my special thanks go to Andr\'as Szenes.

\section{Tautological integrals}
%Ultimate goal: Determine the geometric and topological invariants of $X^{[n]}$. \\
%In this paper we develop a localisation method to compute tautological integrals on moduli of jets of curves and its connection to Hilbert schemes.

%For surfaces these were extensively studied such as \textbf{Betti numbers} (Ellingsrud, Stromme, G\"ottsche), \textbf{Hodge numbers} (S\"orgel, G\"ottsche), \textbf{cohomology ring} (Nakajima, Grojnowski), \textbf{Chern numbers of tautological bundles} (Lehn, Rennemo, Marian-Oprea-Pandharipande).\\

%When $\dim(X)>2$ not much is known: Rennemo, Tzeng, motivic DT invariants (Behrend-Bryan-Szendroi).\\
Let $X$ be a smooth projective variety of dimension $n$ and let $F$ be a rank $r$ bundle (loc. free sheaf) on $X$.  
Let 
\[X^{[k]}=\{\xi \subset X:\dim(\xi)=0,\mathrm{length}(\xi)=\dim H^0(\xi,\calo_\xi)=k\}\]
denote the Hilbert scheme of $k$ points on $X$ parametrizing length $k$ subschemes of $X$ and $F^{[k]}$ the corresponding rank $rk$ bundle on $X^{[k]}$ whose fibre over $\xi \in X^{[k]}$ is $F \otimes \calo_{\xi}=H^0(\xi,F|_\xi)$. 
  
 Equivalently, $F^{[k]}=q_*p^*(F)$ where $p,q$ denote the projections from the universal family of subschemes $\mathcal{U}$ to $X$ and $X^{[k]}$ respectively: 
 \[X^{[k]}\times X \supset \xymatrix{\mathcal{U} \ar[r]^-q \ar[d]^-p & X^{[k]} \\ X & }.\]
  
For simplicity let $\Hilb^k_0(\CC^n)$ denote the punctual Hilbert scheme of $k$ points on $\CC^n$ defined as the closed subset of $\Hilb^k(\CC^n)$ parametrising subschemes supported at the origin. Following Rennemo \cite{rennemo} we define punctual geometric subsets to be the constructible subsets of the punctual Hilbert scheme containing all $0$-dimensional schemes of given isomorphism types.
\begin{definition} A punctual geometric set is a constructible subset $Q\subseteq \Hilb^k_0(\CC^n)$ which is the union of isomorphism classes of subschemes, that is, if $\xi \in Q$ and $\xi' \in \Hilb^k_0(\CC^n)$ are isomorphic schemes then $\xi' \in Q$. 
\end{definition}

\begin{definition} For an $s$-tuple $(Q_1,\ldots, Q_s)$ of punctual geometric sets such that $Q_i\subseteq \Hilb^{k_i}_0(\CC^n)$ and $k=\sum k_i$ define
\[P(Q_1,...,Q_s)=\{\xi \in X^{[k]}:\xi=\xi_1 \sqcup  \ldots \sqcup \xi_s \text{ where } \xi_i\in X^{[k_i]}_{p_i}\cap Q_i \text{  for distinct } p_1,\ldots, p_s\}\subseteq X^{[k]}.\]
A subset $\calz \subseteq X^{[k]}$ is geometric if it can be expressed as finite union, intersection and complement of sets of the form $P(Q_1,\ldots ,Q_s)$.
\end{definition}

A straightforward way to produce punctual geometric subsets is by taking a complex algebra $A$ of complex dimension $k$ and define the corresponding 
\[Q_A=\{\xi \in X^{[k]}: \calo_\xi \simeq A\}.\]  
When $A=\CC[z]/z^{k}$ then $Q_A=CX^{[k]}_p$ is the punctual curvilinear locus defined in the next section and   
\[\overline{CX}^{[k]}=\cup_{p\in X} \overline{CX}^{[k]}_p\]
 is the curvilinear Hilbert scheme, the central object of this paper.  

In this paper we work with singular homology and cohomology with rational coefficients. For a smooth manifold $X$ the degree of a class $\eta \in H_*(X)$ means its pushforward to $H_*(pt)=\QQ$. By choosing  $\alpha_\eta \in \Omega^{top}(X)$,  a closed compactly supported differential form representing the cohomology class $\eta$ this degree is equal to the integral  
\[\mu \cap [X]=\int_{X} \mu.\]

Let $\calz \subset X^{[k]}$ be a geometric subset with closure $\overline{\calz}$ and $M(c_1,\ldots, c_{rk})$ be a monomial in the Chern classes $c_i=c_i(F^{[k]})$ of weighted degree equal to $\dim \overline{\calz}$ where the weight of $c_i$ is $2i$. By choosing $\alpha_M \in \Omega^*(X^{[k]})$, a closed compactly supported differential form representing the cohomology class of $M(c_1,\ldots, c_{rk})$, the degree  
\[[\overline{\calz}]\cap M(c_1,\ldots, c_{rk})=\int_{\overline{\calz}} \alpha_M\]
is called a tautological integral of $F^{[k]}$. Here the integral of $\alpha_M$ on the smooth part of $\overline{\calz}$ is absolutely convergent and by definition we denote this by $\int_{\overline{\calz}} \alpha_M$.

\begin{remark}\label{remark:pullback}
Recall (see e.g. \cite{botttu}) that if $f:X\to Y$ is a smooth proper
map between connected oriented manifolds such that $f$ restricted to
some open subset of $X$ is a diffeomorphism, then for a compactly
supported form $\mu$ on $Y$, we have $\int_X   f^*\mu= \int_Y\mu$.
The analogous statement for singular varieties is the following. Let $f:M\to N$ be a 
smooth proper map between smooth quasiprojective
varieties and assume that $X\subset M$ and $Y\subset N$ are possibly
singular closed subvarieties, such that $f$ restricted
to $X$ is a birational map from $X$ to $Y$. If $\mu$ is a closed differential form on $N$ then the integral of $\mu$ on the smooth part of
$Y$ is absolutely convergent; we denote this by $\int_Y\mu$. With this
convention we again have
$\int_X   f^*\mu= \int_Y\mu$.

In particular this means means that the integral $\int_Y \mu$ of the compactly supported form $\mu$ on $N$ is the same as the integral $\int_{\tilde{Y}}f^*\mu$ of the pull-back form $f^*\mu$ over any (partial) resolution $f:(\tilde{Y},\tilde{M}) \to (Y,M)$.

\end{remark}
 
 In \S\ref{sec:resolutions} we construct an embedding $\overline{CX^{[k+1]}}_p \subset \grass_k(\symdot)$ into a smooth Grassmannian and for $k\le n$ we construct a partial resolution $\widetilde{CX}^{[k]}_p \to \overline{CX}^{[k]}_p$. In \S\ref{sec:locsnowman} we develop the iterated residue formula of Theorem \ref{main} using equivariant localisation to compute $\int_{\widetilde{CX}^{[k]}}P(c_i(F^{[n]}))$ which is according to the remark above equal to $\int_{\overline{CX}^{[k]}}P(c_i(F^{[n]}))$.

\section{Curvilinear Hilbert schemes}\label{sec:curvilinear}
In this section we describe a geometric model for curvilinear Hilbert schemes. Let $X$ be a smooth projective variety of dimension $n$ and let  
\[X^{[k]}=\{\xi \subset X:\dim(\xi)=0,\mathrm{length}(\xi)=\dim H^0(\xi,\calo_\xi)=k\}\]
denote the Hilbert scheme of $k$ points on $X$ parametrizing all length $k$ subschemes of $X$. For $p\in X$ let 
\[X^{[k]}_p=\{ \xi \in X^{[k]}: \mathrm{supp}(\xi)=p\}\]
denote the punctual Hilbert scheme consisiting of subschemes supported at $p$. If $\rho: X^{[k]} \to S^kX$, $\xi \mapsto \sum_{p\in X}\mathrm{length}(\calo_{\xi,p})p$ denotes the Hilbert-Chow morphism then $X^{[k]}_p=\rho^{-1}(kp)$.
\begin{definition}
A subscheme $\xi \in X^{[k]}_p$ is called curvilinear if $\xi$ is contained in some smooth curve $C\subset X$. Equivalently, one might say that $\calo_\xi$ is isomorphic  to the $\CC$-algebra $\CC[z]/z^{k}$.
The punctual curvilinear locus at $p\in X$ is the set of curvilinear subschemes supported at $p$: 
\[CX^{[k]}_p=\{\xi \in X^{[k]}_p:\xi \subset \mathcal{C}_p \text{ for some smooth curve } \mathcal{C} \subset X\}=\{\xi \in X^{[k]}_p:\calo_\xi \simeq \CC[z]/z^{k}\}.\]
\end{definition}

For surfaces ($n=2$) $CX^{[k]}_p$ is an irreducible quasi-projective variety of dimension $n-1$ which is an open dense subset in $X^{[k]}_p$ and therefore its closure is the full punctual Hilbert scheme at $p$, that is, $\overline{CX}^{[k]}_p=X^{[k]}_p$. When $n\ge 3$ the punctual Hilbert scheme $X^{[k]}_p$ is not necessarily irreducible or reduced, but the closure of the curvilinear locus is one of its irreducible components:  

\begin{lemma} $\overline{CX^{[k]}_p}$ is an irreducible component of the punctual Hilbert scheme $X^{[k]}_p$ of dimension $(n-1)(k-1)$. 
\end{lemma}

\begin{proof}
Note that $\xi \in \Hilb^{[k]}_0(\CC^n)$ is not curvilinear if and only if $\calo_\xi$ does not contain elements of degree $k$, that is, after fixing some local coordinates $x_1,\ldots, x_n$ of $\CC^n$ at the origin we have
\[\calo_\xi \simeq \CC[x_1,\ldots, x_n]/I \text{ for some } I\supseteq (x_1,\ldots, x_n)^k.\]
This is a closed condition and therefore curvilinear subschemes can't be approximated by non-curvilinear subschemes in $\Hilb^{[k]}_0(\CC^n)$. The dimension of $\overline{CX^{[k]}_p}$ will come from the description of it as a non-reductive quotient in the next subsection.  
\end{proof}

Note that any curvilinear subscheme contains only one subscheme for any given smaller length and any small deformation of a curvilinear subscheme is again locally curvilinear.

\begin{remark}
Fix coordinates $x_1,\ldots, x_n$ on $\CC^n$. Recall that the defining ideal $I_\xi$ of any subscheme $\xi \in \Hilb^{k+1}_0(\CC^n)$ is a codimension $k$ subspace in the maximal ideal $\mathfrak{m}=(x_1,\ldots, x_n)$. The dual of this is a $k$-dimensional subspace $S_\xi$ in $\mathfrak{m}^*\simeq \symdot$ giving us a natural embedding $\varphi: X^{[k+1]}_p \hookrightarrow \grass_k(\symdot)$. In what follows, we give an explicit  parametrization of this embedding using an algebraic model coming from global singularity theory.  
\end{remark}

\subsection{Test curve model for $CX^{[k]}_p$}\label{subsec:testcurve}
If $u,v$ are positive integers let $J_k(u,v)$ denote the vector
space of $k$-jets of holomorphic maps $(\CC^u,0) \to (\CC^v,0)$ at
the origin, that is, the set of equivalence classes of maps
$f:(\CC^u,0) \to (\CC^v,0)$, where $f\sim g$ if and only if
$f^{(j)}(0)=g^{(j)}(0)$ for all $j=1,\ldots ,k$.

If we fix local coordinates $z_1,\ldots, z_u$ at $0\in \CC^u$ we can
again identify the $k$-jet of $f$ with the set of derivatives at the
origin, that is $(f'(0),f''(0),\ldots, f^{(k)}(0))$, where
$f^{(j)}(0)\in \mathrm{Hom}(\mathrm{Sym}^j\CC^u,\CC^v)$. This way we
get the equality
\begin{equation}\label{identification}
J_k(u,v)=\oplus_{j=1}^k\mathrm{Hom}(\mathrm{Sym}^j\CC^u,\CC^v).
\end{equation}
One can compose map-jets via substitution and elimination of terms
of degree greater than $k$; this leads to the composition maps
\begin{equation}
  \label{comp}
\jetk uv \times \jetk vw \to \jetk uw,\;\;  (\Psi_1,\Psi_2)\mapsto
\Psi_2\circ\Psi_1 \mbox{modulo terms of degree $>k$ }.
\end{equation}
When $k=1$, $J_1(u,v)$ may be identified with $u$-by-$v$ matrices,
and \eqref{comp} reduces to multiplication of matrices.

The $k$-jet of a curve $(\CC,0) \to (\CC^n,0)$ is simply an element
of $J_k(1,n)$. We call such a curve $\gamma$ {\em regular}, if
$\gamma'(0)\neq 0$; introduce the notation $\jetreg 1n$ for the set
of regular curves:
\[\jetreg 1n=\left\{\g \in \jetk 1n; \g'(0)\neq 0 \right\}\]
Let $\xi \in CX^{[k+1]}_p$ be a curvilinear subscheme. It is contained in a unique smooth curve germ $\mathcal{C}_p$ 
\[\xi \subset \mathcal{C}_p \subset X.\]
Let $f_{\xi}:(\CC,0)\to (X,p)$ be a $k$-jet of germ parametrising $\mathcal{C}_p$. 
Then $f_{\xi}\in \jetreg 1n$ is determined only up to polynomial reparametrisation germs $\phi: (\CC,0)\to (\CC,0)$ and therefore we get  
\begin{lemma} The curvilinear locus $CX^{[k+1]}_p$ is equal (as a set) to the set of $k$-jet of regular germs at the origin of $\CC^n$ modulo reparametrisation: 
\[CX^{[k+1]}_p=\{k\text{-jets } (\CC,0)\to (\CC^n,0)\}/\{k\text{-jets } (\CC,0)\to (\CC,0)\}=\jetreg 1n/J_k^{\reg}(1,1).\]
\end{lemma}

We can explicitely write out this reparametrisation action as follows; let $f_{\xi}(z)=z f'(0)+\frac{z^2}{2!}f''(0)+\ldots +\frac{z^k}{k!}f^{(k)}(0) \in \jetreg 1n$ a $k$-jet of germ at the origin (i.e no constant term) in $\CC^n$ with $f^{(i)}\in \CC^n$ such that $f' \neq 0$ and let  
$\varphi(z)=\alpha_1z+\alpha_2z^2+\ldots +\alpha_k z^k \in \jetreg 11$ with $\alpha_i\in \CC, \alpha_1\neq 0$.  
 Then 
 \[f \circ\varphi(z)
=(f'(0)\alpha_1)z+(f'(0)\alpha_2+
\frac{f''(0)}{2!}\alpha_1^2)z^2+\ldots
+\left(\sum_{i_1+\ldots +i_l=k}
\frac{f^{(l)}(0)}{l!}\alpha_{i_1}\ldots \alpha_{i_l}\right)z^k=\]
\begin{equation}\label{jetdiffmatrix}
=(f'(0),\ldots, f^{(k)}(0)/k!)\cdot 
\left(
\begin{array}{ccccc}
\alpha_1 & \alpha_2   & \alpha_3          & \ldots & \alpha_k \\
0        & \alpha_1^2 & 2\alpha_1\alpha_2 & \ldots & 2\alpha_1\alpha_{k-1}+\ldots \\
0        & 0          & \alpha_1^3        & \ldots & 3\alpha_1^2\alpha_ {k-2}+ \ldots \\
0        & 0          & 0                 & \ldots & \cdot \\
\cdot    & \cdot   & \cdot    & \ldots & \alpha_1^k
\end{array}
 \right)
 \end{equation}
where the $(i,j)$ entry is $p_{i,j}(\bar{\alpha})=\sum_{a_1+a_2+\ldots +a_i=j}\alpha_{a_1}\alpha_{a_2} \ldots \alpha_{a_i}.$

\begin{remark}\label{naturalembedding}
The linearisation of the action of $\jetreg 11$ on $\jetreg 1n$ given as the matrix multiplication in \eqref{jetdiffmatrix} represents $\jetreg 11$ as a upper triangular matrix group in $\GL(n)$. It is parametrised along its first row with the free parameters $\alpha_1,\ldots, \alpha_k$ and the other entries are certain (weighted homogeneous) polynomials in these free parameters. It is a $\CC^*$ extension of its maximal  unipotent radical
\[\jetreg 11=U \rtimes \CC^*\]
where $U$ is the subgroup we get via substituting $\alpha_1=1$ and the diagonal $\CC^*$ acts with weights $0,1\ldots, n-1$ on the Lie algebra $\mathrm{Lie}(U)$. In \cite{bk} and \cite{bdhk} we study actions of groups of this type in a more general context. 
\end{remark}

Fix an integer $N\ge 1$ and define
\[\Theta_k=\left\{\Psi\in J_k(n,N):\exists \g \in \jetreg 1n: \Psi \circ \g=0
\right\},\]
that is, $\Theta_k$ is the set of those $k$-jets of germs on $\CC^n$ at the origin which vanish on some regular curve. By definition, $\Theta_k$ is the image
of the closed subvariety of $\jetk nN \times \jetreg 1n$ defined by
the algebraic equations $\Psi \circ \g=0$, under the projection to
the first factor. If $\Psi \circ \gamma=0$, we call $\g$ a {\em test
curve} of $\Theta$. 

\begin{remark}
The subset $\Theta_k$ is the closure of an important singularity class in the jet space $J_k(n,N)$. These are called Morin singularities and the equivariant dual of $\Theta_k$ in $J_k(n,N)$ is called the Thom polynomial of Morin singularities, see \cite{bsz} for details.
\end{remark}

Test curves of germs are generally not unique. A basic but crucial observation is the following. If $\g$ is a test
curve of $\Psi \in \Theta_k$, and $\vp \in \jetreg 11$ is a 
holomorphic reparametrisation of $\CC$, then $\g \circ \vp$ is,
again, a test curve of $\Psi$:
\begin{displaymath}
\label{basicidea}
\xymatrix{
  \CC \ar[r]^\vp & \CC \ar[r]^\g & \CC^n \ar[r]^{\Psi} & \CC^N}
\end{displaymath}
\[\Psi \circ \g=0\ \ \Rightarrow \ \ \ \Psi \circ (\g \circ \vp)=0\]

In fact, we get all test curves of $\Psi$ in this way if the
following open dense property holds: the linear part of $\Psi$ has
$1$-dimensional kernel. Before stating this in Theorem 
\ref{embedgrass} below, let us write down the equation $\Psi \circ
\g=0$ in coordinates in an illustrative case. Let
$\g=(\g',\g'',\ldots, \g^{(k)})\in \jetreg 1n$ and
$\Psi=(\Psi',\Psi'',\ldots, \Psi^{(k)})\in \jetk nN$ be the
$k$-jets of the test curve $\g$ and the map $\Psi$ respectively. Using the chain rule and the notation $v_i=\g^{(i)}/i!$, the equation $\Psi \circ \g=0$ reads
as follows for $k=4$:
\begin{eqnarray} \label{eqn4}
& \Psi'(v_1)=0, \\ \nonumber & \Psi'(v_2)+\Psi''(v_1,v_1)=0, \\
\nonumber
& \Psi'(v_3)+2\Psi''(v_1,v_2)+\Psi'''(v_1,v_1,v_1)=0, \\
&
\Psi'(v_4)+2\Psi''(v_1,v_3)+\Psi''(v_2,v_2)+
3\Psi'''(v_1,v_1,v_2)+\Psi''''(v_1,v_1,v_1,v_1)=0.
\nonumber
\end{eqnarray}

%To simplify our formulas we introduce the following notations for a
%partition $\tau=[i_1\ldots i_l]$ of the integer $i_1+\ldots +i_l$:
%\begin{itemize}
%\item
% the
%\textsl{length}: $|\tau|=l$,
%\item  the \textsl{sum}: $\sum\tau=i_1+\ldots +i_l$,
%\item \textsl{number of permutations}: $\comb(\tau)$, which is
 % the number of different sequences consisting of the numbers
%  $i_1,\dots, i_l$; e.g. $\comb([1,1,1,3])=4$.
%\item $\bg_\tau = \prod_{j=1}^l \g^{(i_j)}\in\Sym^l\CC^n\;\text{ and }\;
%\Psi(\bg_\tau)=\Psi^l(\g^{(i_1)},\dots,\g^{(i_l)})\in\CC^N$.
%\end{itemize}

\begin{lemma}[\cite{gaffney,bsz}]\label{explgp} Let
$\g=(\g',\g'',\ldots, \g^{(k)})\in \jetreg 1n$ and
$\Psi=(\Psi',\Psi'',\ldots, \Psi^{(k)})\in \jetk nN$ be $k$-jets.
Then substituting $v_i=\g^{(i)}/i!$, the equation $\Psi\circ \g$ is equivalent to
  the following system of $k$ linear equations with values in
  $\CC^N$:
\begin{equation}
  \label{modeleq}
\sum_{\tau \in \mathcal{P}(m)} \Psi(\bv_\tau)=0,\quad m=1,2,\dots, k.
\end{equation}
Here $\mathcal{P}(m)$ denotes the set of partitions $\tau=1^{\tau_1}\ldots m^{\tau_m}$ of $m$ into nonnegative integers and $\bv_\tau=v_1^{\tau_1}\cdots v_{m}^{\tau_m}$. 
\end{lemma}
For a given $\g \in \jetreg 1n$ and $1\le i \le k$ let $\mathcal{S}^{i,N}_{\g}$ denote the set of 
solutions of the first $i$ equations in \eqref{modeleq}, that is,
\begin{equation}\label{solutionspace}
\mathcal{S}^{i,N}_\g=\left\{\Psi \in \jetk nN, \Psi \circ \g=0 \text{ up to order } i \right\}
\end{equation}
The equations \eqref{modeleq} are linear in $\Psi$, hence
\[\mathcal{S}^{i,N}_\g \subset \jetk nN\]
is a linear subspace of codimension $iN$, i.e a point of $\grass_{\mathrm{codim}=iN}(J_k(n,N))$, whose dual, $(\mathcal{S}^{i,N}_{\g})^*$, is an $iN$-dimensional subspace of $J_k(n,N)^*$. These subspaces are invariant under the reparametrization of $\g$. In fact, $\Psi \circ \gamma$ has $N$ vanishing coordinates and therefore
\[\mathcal{S}^{i,N}_{\g}=\mathcal{S}^{i,1}_{\g} \otimes \CC^N\]
holds. 

For $\Psi \in J_k(n,N)$ let $\Psi^1 \in \Hom(\CC^n,\CC^N)$ denote the linear part. When $N\ge n$ then the subset 
\[\tilde{\mathcal{S}}^{i,N}_{\g}=\{\Psi \in \mathcal{S}^{i,N}_{\g}: \dim \ker \Psi^1=1\}\]
is an open dense subset of the subspace $\mathcal{S}^{i,N}_{\g}$. In fact it is not hard to see that the complement $\tilde{\mathcal{S}}^{i,N}_{\g}\setminus \mathcal{S}^{i,N}_{\g}$where the kernel of $\Psi^1$ has dimension at least two is a closed subvariety of codimension $N-n+2$.

Note that for $N=1$, according to \eqref{identification}, the dual space $J_k(n,1)^*$ can be and will be identified with 
\[\Hom(\CC,\mathrm{Sym}^{\le n}(\CC^n)^*)\simeq \symdot\]
where $\symdot=\bigoplus_{i=1}^k \sym^i \CC^n$ and we identified $\CC^n$ with its dual. 
%Furthermore, for $\g \in J_k(1,n)$ by putting $\g^{(i)}/i! \in \CC^n$ into the $i$th column of a matrix we can identify elements of $\jetreg 1n$ as $k$-by-$n$ matrices, that is, elements of $\Hom(\CC^k,\CC^n)$ with nonzero first column.   

\begin{theorem}\label{embedgrass} \begin{enumerate}
\item The map
\[\phi: \jetreg 1n \rightarrow \grass_k(\symdot)\]
defined as  $\gamma  \mapsto (\mathcal{S}^{k,1}_\g)^*$
is $\jetreg 11$-invariant and induces an injective map on the $\jetreg 11$-orbits into the Grassmannian 
\[\phi^\grass: CX^{[k+1]}_p=\jetreg 1n /\jetreg 11 \hookrightarrow \grass_k(\symdot).\]
Moreover, $\phi$ and $\phi^\grass$ are $\GL(n)$-equivariant with respect to the standard action of $\GL(n)$ on $\jetreg 1n \subset \Hom(\CC^k,\CC^n)$ and the induced action on $\grass_k(\symdot)$.
\item The image of $\phi$ and the image of $\varphi$ defined in Remark \ref{naturalembedding} coincide in $\grass_k(\symdot)$:
\[\mathrm{im}(\phi)=\mathrm{im}(\varphi)\subset \grass_k(\symdot).\]
\end{enumerate}
\end{theorem}

\begin{proof}
For the first part it is enough to prove that for $\Psi \in \Theta_k$ with $\dim \ker \Psi^1=1$ and $\gamma,\delta \in \jetreg 1n$
\[\Psi \circ \gamma_1=\Psi \circ \gamma_2=0 \Leftrightarrow \exists
\Delta \in \jetreg 11 \text{ such that } \gamma=\delta
\circ \Delta.\]
We prove this statement by induction. Let $\gamma=v_1t+\dots +v_kt^k$ and
$\delta=w_1t+\dots+ w_kt^k$. Since $\dim \ker \Psi^1=1$, $v_1=\lambda w_1$, for some
$\lambda\neq0$. This proves the $k=1$ case. 

Suppose the statement is true for $k-1$. Then, using the appropriate
order-($k-1$) diffeomorphism, we can assume that $v_m=w_m$, $m=1\ldots
k-1.$ It is clear then from the explicit form \eqref{modeleq}
(cf. \eqref{eqn4}) of the equation
$\Psi\circ\gamma=0$, that  $\Psi^1(v_k)=\Psi^1(w_k)$, and hence
$w_k=v_k-\lambda v_1$ for some $\lambda\in\CC$. Then
$\gamma=\Delta\circ\delta$ for $\Delta=t+\lambda t^k$, and the proof
is complete.

 The second part immediately follows from the definition of $\varphi$ and $\phi$. 
\end{proof}

\begin{remark} \begin{enumerate}
\item In particular the second part of Theorem \ref{embedgrass} tells us that the curvilinear component
\[\overline{CX^{[k+1]}_p}=\overline{\mathrm{im}(\phi)} \subset \grass_k(\symdot)\]
has a $\GL(n)$-equivariant embedding into the Grassmannian $\grass_k(\symdot)$ as the closure of the image of $\phi$. 
\item  For a point $\gamma\in J_k^\reg(1,n)$ let $v_i=\frac{\g^{(i)}}{i!}\in \CC^n$ denote the normed $i$th derivative. Then from Lemma \ref{explgp} immediately follows that for $1\le i \le k$ (see \cite{bsz}):
\begin{equation}\label{sgamma}
\mathcal{S}^{i,1}_\g=\mathrm{Span}_\CC (v_1,v_2+v_1^2,\ldots, \sum_{\tau \in \mathcal{P}(i)}\bv_{\tau})\subset \symdot.
\end{equation} 
This explicit parametrisation of the curvilinear component is crucial in building our localisation process in the next section.
\item Since $\phi$ is $\GL(n)$-equivariant, for $k\le n$ the $\GL(n)$-orbit of 
\[p_{n,k}=\phi(e_1,\ldots, e_k)=\mathrm{Span}_\CC (e_1,e_2+e_1^2,\ldots, \sum_{\tau \in \mathcal{P}(k)}e_{\tau}),\] 
forms a dense subset of the image $\jetreg 1n$ and therefore 
\[\overline{CX^{[k+1]}_p}=\overline{\phi(\jetreg 1n)}=\overline{\mathrm{GL_n} \cdot p_{n,k}}.\] 
\end{enumerate}
\end{remark}

\subsection{Jet bundles and $CX^{[k]}$.} Let $X$ be a smooth projective variety and let $J_kX \to X$ denote the bundle of $k$-jets of germs of parametrized curves in $X$; its
fibre over $x\in X$ is the set of equivalence classes of germs of holomorphic
maps $f:(\CC,0) \to (X,x)$, with the equivalence relation $f\sim g$
if and only if the derivatives $f^{(j)}(0)=g^{(j)}(0)$ are equal for
$0\le j \le k$. If we choose local holomorphic coordinates
$(z_1,\ldots, z_n)$ on an open neighbourhood $\Omega \subset X$
around $x$, the elements of the fibre $J_kX_x$ are represented by the Taylor expansions 
\[f(t)=x+tf'(0)+\frac{t^2}{2!}f''(0)+\ldots +\frac{t^k}{k!}f^{(k)}(0)+O(t^{k+1}) \]
 up to order $k$ at $t=0$ of $\CC^n$-valued maps $f=(f_1,f_2,\ldots, f_n)$
on open neighbourhoods of 0 in $\CC$. Locally in these coordinates the fibre  can be written as
\[J_kX_x=\left\{(f'(0),\ldots, f^{(k)}(0)/k!)\right\}=(\CC^n)^k,\]
which we identify with $J_k(1,n)$.  Note that $J_kX$ is not
a vector bundle over $X$ since the transition functions are polynomial but
not linear, see \cite{dem} for details. 

Let $J_k^{\reg}X$ denote the bundle of $k$-jets of germs of parametrized regular curves in $X$, that is, where the first derivative $f'\neq 0$ is nonzero. Its fibre is isomorphic with $\jetreg 1n$.

$\jetreg 11$ acts fibrewise on the jet bundle $J_k^{\reg}X$ and the full curvilinear component $CX^{[k]}$ on $X$ can be identified with the non-reductive fibrewise quotient of $J_k^{\reg}X$ by $\jetreg 11$:
\[CX^{[k+1]}=J_k^{\reg}X/\jetreg 11.\]
More precisely, introduce the notation
\[\symdotx=T_X^*\oplus \sym^2(T_X^*)\oplus \ldots \oplus \sym^k(T_X^*)\]
for the vector bundle on $X$ whose fibre is isomorphic to $\symdot$.
The Grassmannian bundle $\grass_k(\symdotx)$ and the jet bundle $J_k^{\reg}X$ have an induced fibrewise action of $\GL(n)$ and we have the following fibrewise version of Theorem \ref{embedflag} 
\begin{corollary}
The quotient $J_k^{\reg}X/\jetreg 11$ has the structure of a locally trivial
bundle over $X$, and Theorem \ref{embedflag} gives us a $\GL(n)$-equivariant holomorphic embedding
\[\phi^\grass :CX^{[k+1]}=J_k^{\reg}X/\jetreg 11 \hookrightarrow \grass_k(\symdotx)\] into the Grassmannian bundle of $\symdotx$ over $X$. The fibrewise compactification  
\[\overline{CX}^{[k+1]}=\overline{\phi^\grass(J_k^{\reg}X)}\]
 of the image is the curvilinear component of the Hilbert scheme of $k+1$ points on $X$.
\end{corollary}

\subsection{Tautological bundles over $\overline{CX}^{[k]}$} Let $F$ be a rank $r$ vector bundle over $X$. The fibre of the corresponding rank $r(k+1)$ tautological bundle $F^{[k+1]}$ on $\overline{CX}^{[k+1]}$ at the point $\xi$ is 
\[F^{[k+1]}_{\xi}=H^0(\xi, F|_{\xi})=H^0(\calo_{\xi} \otimes F).\]
Using our embedding $\phi^\grass: \overline{CX}^{[k+1]} \hookrightarrow \grass_{k}(\symdotx)$ this fibre can be identified as
\[F^{[k+1]}_{\xi}=(\calo_{\grass_{k}(\symdotx)} \oplus \cale)_{\phi(\xi)} \otimes F_{\mathrm{supp}(\xi)}\]
where $\cale$ is the tautological rank $k$ bundle over $\grass_k(\symdotx)$. Hence the total Chern class of $F^{[k+1]}$ can be written as 
\[c(F^{[k+1]})=\prod_{j=1}^r(1+\theta_j)\prod_{i=1}^k\prod_{j=1}^r(1+\eta_i+\theta_j)\]
where $c(F)=\prod_{j=1}^r(1+\theta_j)$ and $c(\cale)=\prod_{i=1}^k(1+\eta_i)$ are the Chern classes for the corresponding bundles. In particular the Chern class
\begin{equation}\label{chernclasses}
c_i(F^{[k+1]})=\mathcal{C}_i(c_1(\cale),\ldots c_k(\cale),c_1(F),\ldots, c_r(F))
\end{equation}
can be expressed as a polynomial function $\mathcal{C}_i$ in Chern classes of $\cale$ and $F$.  
\section{Partial resolutions of $\overline{CX}^{[k+1]}$}\label{sec:resolutions}
In this section first we construct a partial resolution of the (highly singular) punctual curvilinear component $\overline{CX}^{[k+1]}_p \subset \grass_k(\symdot)$ in two steps. The first partial resolution is defined for any choice of parameters $n,k$ an it uses nested Hilbert schemes. For the second step we need to impose the very restrictive condition $k\le n$, that is the number of points can't exceed the dimension of the variety plus $1$. We will see how to dispose this condition in Section \S\ref{sec:extendtherange}.

\subsection{Completion in nested Hilbert schemes} Let 
\[X^{[k_1,\ldots, k_t]}=\{(\xi_1 \subset \xi_2 \subset \ldots \subset \xi_t): \xi_i \in X^{[k_i]}\} \subset X^{[k_1]} \times \ldots \times X^{[k_t]}\]
be the nested Hilbert scheme defining flags of subschemes of length vector $(k_1,\ldots, k_t)$.

Curvilinear subschemes contain only one subscheme for any given smaller length. Therefore $\xi \in CX^{[k+1]}_p$ defines a unique flag 
\[\calf(\xi)=(\xi_1 \subset \xi_2 \subset \ldots \subset \xi_k) \in CX^{[2]}_p \times \ldots \times CX^{[k+1]}_p\subset  X^{[2,\ldots k+1]}\]
where $\xi_i$ is defined via
\[\calo_{\xi_i}=\calo_\xi /\calo_{X,p}^{i+1} \simeq \CC[z]/z^{i+1}\]
and therefore $\xi_i \in CX^{[i+1]}_p$. This defines an embedding 
\[\tilde{\phi}: CX^{[k+1]}_p \hookrightarrow X^{[2,\ldots, k+1]}\]
\[\xi \mapsto (\xi_1 \subset \ldots \subset \xi_k).\]
Let $f_\xi \in \jetreg 1n$ denote the $k$-jet corresponding to $\xi \in CX^{[k+1]}_p$ and let $\mathcal{S}^i_\xi=\mathcal{S}^{i,1}_{f_\xi} \subset J_k(n,1)$ be the solution space defined in \eqref{solutionspace} where $N=1$. Then $\tilde{\phi}$ can be equivalently written as 
\[ f_\xi \mapsto ((\mathcal{S}^1_\xi)^*\subset (\mathcal{S}^2_\xi)^* \subset \ldots \subset (\mathcal{S}^k_\xi)^*) \in \flag_k(\symdot)\]
or using coordinates as 
\[f=f_\xi \mapsto \mathrm{Span}_\CC(f') \subset \mathrm{Span}_{\CC}(f',f''+(f')^2) \subset \ldots \subset \mathrm{Span}_{\CC}(f',f''+(f')^2,\ldots ,f^{[k]}+\sum_{\Sigma a_i=k}(f^{[i]})^{a_i}).\]
Theorem \ref{embedgrass} has the following immediate 
 \begin{corollary}\label{embedflag} The map 
\[\tilde{\phi}: \jetreg 1n \rightarrow \flag_k(\symdot)
\]
\[\gamma  \mapsto \Fl_\g=((\mathcal{S}^1_{\gamma})^*\subset \ldots \subset (\mathcal{S}^k_{\gamma})^*))\]
is $\jetreg 11$-invariant and induces an injective map on the $\jetreg 11$-orbits into the flag manifold 
\[\phi^\flag: CX^{[k+1]}_p=\jetreg 1n /\jetreg 11 \hookrightarrow \flag_k(\symdot).\]
Moreover, all these maps are $\GL(n)$-equivariant with respect to the standard action of $\GL(n)$ on $\jetreg 1n \subset \Hom(\CC^k,\CC^n)$ and the induced action on $\flag_k(\symdot)$.
\end{corollary}

Let $\widehat{CX}^{[k+1]}_p$ denote the closure of $\tilde{\phi}(\jetreg 1n)$ in  $\mathrm{Flag}(k,\symdot)$.\\
\subsection{Blowing up along the linear part.}\label{subsec:blowingup} 
Assume $k\le n$. Let $\pi: J_k(n,1)^* \simeq \symdot=\oplus_{i=1}^k\mathrm{Sym}^i\CC^n \to \CC^n$ denote the projection to the first (linear) factor and define  
\[\widetilde{CX}^{[k+1]}_p=\{((\mathcal{S}^1_{\gamma})^*\subset \ldots \subset (\mathcal{S}^k_{\gamma})^*)),(V_1\subset \ldots \subset V_k): \pi(\mathcal{S}^i_{\gamma})^*) \subset V_i\} \subset \widehat{CX}^{[k+1]}_p \times \flag_k(\CC^n).\]
 Equivalently, let $P_{k,n} \subset \mathrm{GL}_n$ denote the parabolic subgroup which preserves the flag 
\[\mathbf{f}=(\mathrm{Span}(e_1)   \subset \mathrm{Span}(e_1,e_2) \subset \ldots \subset \mathrm{Span}(e_1,\ldots, e_k) \subset \CC^n).\] 
 and $\mathfrak{p}_{k,n}=\tilde{\phi}(e_1,\ldots ,e_k)$ the base point in $\flag_k(\symdot)$. Define the partial resolution  $\widetilde{CX}^{[k+1]}_p$ of $\widehat{CX}^{[k+1]}_p$ as the fibrewise compactification of $\widehat{CX}^{[k+1]}_p$ on $\flag_k(\CC^n)=\GL(n)/P_{k,n}$, that is, 
\[\widetilde{CX}^{[k+1]}_p=\GL(n) \times_{P_{k,n}} \overline{P_{k,n}\cdot \mathfrak{p}_{k,n}} \to \overline{\GL(n)\cdot \mathfrak{p}_{k,n}}=\widehat{CX}^{[k+1]}_p\]
with the resolution map $\widetilde{CX}^{[k+1]}_p \to \widehat{CX}^{[k+1]}_p$ given by $\rho(g,z)=g\cdot z$. 

%A third equivalent phrasing of the definition of $\widetilde{CX}^{[k+1]}_p$ is the following. Let $\jetnondeg 1n \subset \jetreg 1n$ be the set of test curves with $\g',\ldots, \g^{(k)}$ linearly independent. These correspond to the regular $n \times k$ matrices in $\Hom(\CC^k,\CC^n)$, and they fibre over the set of complete flags in $\CC^n$:
%\[\jetnondeg 1n \to \Hom(\CC^k,\CC^n)/B_k=\flag_k(\CC^n)\]
%where $B_k \subset \GL(k)$ is the upper Borel. The image of the fibres under $\phi$ are isomorphic to $P_{n,k} \cdot p_k$, and therefore $\tilde{X}_k$ is the fibrewise compactification of $\jetnondeg 1n$ over $\flag_k(\CC^n)$. 

%In \cite{bsz} the authors develop an iterated residue formula based on equivariant localisation on $\tilde{X}_k$ to compute multidegrees of Morin singularity classes. We will generalise this method in the next section to compute cohomological pairings on $\tilde{X}_k$. 

The geometric resolutions $\widetilde{CX}^{[k+1]}_p$ and $\widehat{CX}^{[k+1]}_p$ of $CX^{[k+1]}_p$ constructed in this  section form the fibres of partial resolution bundles $\widehat{CX}^{[k]}$ and $\widetilde{CX}^{[k]}$ over $X$ with partial resolution maps 
\[\widetilde{CX}^{[k]} \to \widehat{CX}^{[k]} \to \overline{CX}^{[k]}\]
where  
\[\widehat{CX}^{[k]}=\overline{\tilde{\phi}(J_k(\cotx))} \subset \flag_k(\symdotx)\]
is the closure of the fibrewise embedding $CX^{[k]}=J_k^{\reg}X/\jetreg 11 \hookrightarrow \flag_k(\symdotx)$ and for $k\le n$ 
\[\widetilde{CX}^{[k+1]}=\{(\cals_1 \subset \ldots \subset \cals_k)),(V_1\subset \ldots \subset V_k): \pi(\cals_i) \subset V_i\} \subset \widehat{CX}^{[k+1]} \times \flag_k(\cotx).\]
Here $\pi: J_k^{\reg}X \simeq \oplus_{i=1}^k\mathrm{Sym}^i \cotx \to \cotx$ denotes again the projection to the first factor. The fibre of $\widetilde{CX}^{[k+1]}$ over $p\in X$ is $\widetilde{CX}^{[k+1]}_p$ and therefore $\widetilde{CX}^{[k+1]}$ canonically sits in $\flag_k(\symdotx)$.

%In the next section, following \cite{bsz}, we develop an equivariant localisation formula on $\widetilde{CX}^{[k+1]}$ to compute topological intersection numbers, leading us to an iterated residue formula.			

\section{Equivariant localisation on $\widetilde{CX}^{[k+1]}$}\label{sec:locsnowman}

Let $F$ be a rank $r$ vector bundle over $X$ and let $F^{[k+1]}$ denote the corresponding rank $(k+1)r$ tautological bundle over $X^{[k+1]}$. We use the same notation $F^{[k+1]}$ for its pull-back along the partial resolution map $\widetilde{CX}^{[k+1]} \to \overline{CX}^{[k+1]}$. Fix a Chern polynomial $P(c_1,\ldots, c_{r(k+1)})$ of weighted degree $\dim \overline{CX}^{[k+1]}=n+(n-1)k$ where $c_i=c_i(F^{[k+1]})$ are the Chern classes of the tautological bundle. In this section we start developing an iterated residue formula for the tautological integral   
$\int_{\widetilde{CX}^{[k]}}P$. This formula is attained via a two-step equivariant localisation process and it is crucially based on a vanishing theorem of residues. 

\subsection{Equivariant de-Rham model and the Atiyah-Bott formula}\label{sec:equiv}

This section is a short introduction to equivariant cohomology and localisation. For
more details, we refer the reader to Section 2 of \cite{bsz} and \cite{getzlervergne}. 

Let $G$ be a compact Lie group with Lie algebra $\mathfrak{g}$ and let $M$ be a $C^\infty$ manifold endowed with the action of $G$. The $G$-equivariant differential forms are defined as differential form valued polynomial functions on $\mathfrak{g}$:
\[\Omega_G^\bullet(M)=\left\{\alpha:\mathfrak{g}\to
\Omega^\bullet(M):\alpha(gX)=g\alpha(X) \text{ for }g\in G , X \in \mathfrak{g}\right\} =(S^\bullet \mathfrak{g}^*\otimes \Omega^\bullet(M))^{G}\]
where $(g\cdot \alpha)(X)=g\cdot (\alpha(g^{-1}\cdot X))$. 
 One can define equivariant the exterior differential $d_G$ on $(S^\bullet \mathfrak{g}^* \otimes \Omega^\bullet(M))^G$ by the formula 
\[(d_G\alpha)(X)=(d-\iota(X_M))\alpha(X),\]
where $\iota(X_M)$ denotes the contraction by the vector field $X_M$. 
This increases the degree of an equivariant form by one if the $\mathbb{Z}$-grading is given on $(S^\bullet \mathfrak{g}^* \otimes \Omega^\bullet(M))^G$ by $\deg(P \otimes \alpha)=2\deg(P)+\deg(\alpha)$ 
for $P \in S^\bullet \mathfrak{g}^*, \alpha \in \Omega^\bullet(M)$. 
The homotopy formula $\iota(X)d+d\iota(X)=\mathcal{L}(X)$ implies that $d_G^2(\alpha)(X)=-\mathcal{L}(X)\alpha(X)=0$
for any $\alpha \in (S^\bullet \mathfrak{g}^* \otimes \Omega^\bullet(M))^G$, and therefore $(d_G,\Omega^\bullet_G(M))$ is a complex. The equivariant cohomology $H_G^*(M)$ of the $G$-manifold $M$ is the cohomology of the complex $(d_G,\Omega^\bullet_G(M))$.
Note that $\alpha \in \Omega^\bullet_G(M)$ is equivariantly closed if and only if 
\[\alpha(X)=\alpha(X)^{[0]}+\ldots +\alpha(X)^{[n]} \text{ such that }
\iota(X_M)\alpha(X)^{[i]}=d\alpha(X)^{[i-2]}.\]
Here $\alpha(X)^{[i]} \in \Omega^i(M)$ is the degree-$i$ part of $\alpha(X) \in \Omega^\bullet(M)$ and $\alpha^{[i]}: \mathfrak{g} \to \Omega^i(M)$ is a polynomial function. 

The equivariant push-forward map $\int_M : \Omega_G(M) \to (S^\bullet \mathfrak{g}^*)^G$ is defined 
by the formula 
\begin{equation}\label{localisationmap}
\left(\int_M \alpha\right)(X)=\int_M \alpha(X)=\int_M \alpha^{[n]}(X)
\end{equation}

When the $n$-dimensional complex torus $T=(\CC^*)^n$ acts on $M$ let $K=U(1)^n$ be its maximal unipotent subgroup and $\liet=\mathrm{Lie}(K)$ its Lie algebra. We define the $T$-equivariant cohomology $H_T^\bullet(M)$ to be the $H_K^\bullet(M)$, the equivariant DeRham cohomology defined by the action of $K$. If $M_0(X)$ is the zero locus of the vector field $X_M$, then the form $\alpha(X)^{[n]}$ is exact outside $M_0(X)$. (see Proposition 7.10 in \cite{getzlervergne}), and this suggests that the integral $\int_M \alpha(X)$ depends only on the restriction $\alpha(X)|_{M_0(X)}$.

\begin{theorem}[Atiyah-Bott \cite{ab}, Berline-Vergne \cite{BV}]\label{abbv} Suppose that $M$ is a compact manifold and $T$ is a complex torus acting smoothly on $M$, and the fixed point set $M^T$ of the $T$-action on M is finite. Then for any cohomology class $\a \in H_T^\bullet(M)$
\[\int_M \alpha=\sum_{f\in M^T}\frac{\a^{[0]}(f)}{\mathrm{Euler}^T(T_fM)}.\]
Here $\mathrm{Euler}^T(T_fM)$ is the $T$-equivariant Euler class of the tangent space $T_fM$, and $\alpha^{[0]}$ is the differential-form-degree-0 part of $\alpha$. 
\end{theorem}

The right hand side in the localisation formula considered in the fraction field of the polynomial ring of $H_T^\bullet (\mathrm{point})=H^\bullet(BT)=S^\bullet \mathfrak{t}^*$ (see more on details in \cite{ab,bgv}). Part of the statement is that the denominators cancel when the sum is simplified.

\subsection{Equivariant Poincar\'e duals and multidegrees}
\label{subsec:epdmult} 

The Atiyah-Bott formula works for holomorphic actions of tori on nonsingular projective varieties. In our case, however, the punctual curvilinear component $\overline{CX}^{[k+1]}_p$ is highly singular at the fixed points so the AB localisation does not apply directly as the equivariant Euler class of the tangent space at a singular fixed point is not well defined. But $\overline{CX}^{[k+1]}_p$ sits in the nonsingular ambient space $\grass_k(\symdot)$ and an intuitive idea would be to put $\mathrm{Euler}^T(T_f\grass_k(\symdot))$ into the denominator on the right hand side which we then compensate in the numerator with some sort of dual of the tangent cone of $\overline{CX}^{[k+1]}_p$ at $f$ sitting in the tangent space of $\grass_k(\symdot)$ at $f$. This idea indeed works and it becomes incarnate in the Rossman formula in \S\ref{subsec:rossman}. 

Let $T=(\CC^*)^n$ be a complex torus with $K=U(1)^n$ its maximal compact subgroup and $\liet=\mathrm{Lie}(K)$ its Lie algebra. Let $M$ be a manifold endowed with a $T$ action. The compactly supported equivariant cohomology groups $ H^\bullet_{K,\mathrm{cpt}}(M)
$ are obtained by restricting the equivariant de Rham complex to compactly supported (or quickly
decreasing at infinity) differential forms.  Clearly $H^\bullet_{K,\mathrm{cpt}}(M) $ is a module over
$H^\bullet_K(M)$. When $M=W$ is an $N$-dimensional
complex vector space, and the action is linear, one has
$H^\bullet_K(W)= S^\bullet\mathfrak{t}^*$ and $ H^\bullet_{K,\mathrm{cpt}}(W) $ is
a free module over $H^\bullet_K(W)$ generated by a single element of
degree $2N$:
\begin{equation}
  \label{thomg}
   H^\bullet_{K,\mathrm{cpt}}(W) = H^\bullet_{K}(W)\cdot\mathrm{Thom}_{K}(W),
\end{equation}
called the Thom class of $W$. 

A $T$-invariant algebraic subvariety $\Sigma$ of dimension $d$ in $W$
represents a $T$-equivariant $2d$-cycle in the sense that
\begin{itemize}
\item a compactly-supported equivariant form $\mu$ of degree $2d$ is
  absolutely integrable over the components of maximal dimension of
  $\Sigma$, and $\int_\Sigma\mu\in S^\bullet \mathfrak{t}$;
\item if $d_K\mu=0$, then $\int_\Sigma\mu$ depends only on the class
  of $\mu$ in $ H^\bullet_{K,\mathrm{cpt}}(W) $,
\item and $\int_\Sigma\mu=0$  if $\mu=d_K\nu$ for a
  compactly-supported equivariant form $\nu$.
\end{itemize}

\begin{definition} \label{defepd} Let $\Sigma$ be an $T$-invariant algebraic
  subvariety of dimension $d$ in the vector space $W$. Then the
  equivariant Poincar\'e dual of $\Sigma$ is the polynomial on $\mathfrak{t}$
  defined by the integral
\begin{equation}
 \label{vergneepd}
 \epd{\Sigma,W} = \frac1{(2\pi)^d}\int_\Sigma\mathrm{Thom}_{K}(W).
\end{equation}  
\end{definition}
 An immediate consequence of the definition is that for an equivariantly
closed differential form $\mu$ with compact support, we have
\[  \int_\Sigma\mu = \int_W \epd{\Sigma,W}\cdot\mu.
\]
This formula serves as the motivation for the term {\em equivariant
  Poincar\'e dual.} This definition naturally extends to the case of an analytic
  subvariety of $\CC^n$  defined in the neighborhood of the origin, or
  more generally, to any $T$-invariant cycle in $\CC^n$.
  
Note that $\epd{\Sigma,W}$ is determined by the maximal dimensional components of $\Sigma$ and in fact it can be characterised and axiomatised by some of its basic properties. These are carefully stated in \cite{bsz} Proposition 2.3 and proofs can be found in \cite{rossmann},\cite{voj},\cite{milsturm}, the list reads as: positivity, additivity on maximal dimensional component, deformation invariance, symmetry and finally a formula for complete intersections of hypersurfaces. These properties provide an algorithm for computing $\epd{\Sigma,W}$ as follows (see \cite{milsturm} \S8.5 and \cite{bsz,b2} for details): we pick any monomial order on the coordinates of $W$ and apply Groebner deformation on the ideal of $\Sigma$ to deform it onto its initial monomial ideal. The spectrum of this monomial ideal is the union of some coordinate subspaces in $W$ with multiplicities whose equivariant dual is then given as the sum of the duals of the maximal dimensional subspaces by the additivity property. For these linear subspaces the formula for complete intersections has the following special form. Let $W=\mathrm{Spec}(\CC[y_1,\ldots, y_N])$ acted on by the $n$-dimensional torus $T$ diagonally where the weight of $y_i$ is $\eta_i$. Then for every subset $\mathbf{i}\subset\{1,\ldots ,
  N\}$ we have
\begin{equation}\label{mdegformula}
\epd{
\{y_{i}=0,\,i\in\mathbf{i}\},W}=
\prod_{i\in\mathbf{i}}\eta_i.
\end{equation}
The weights $\eta_1,\ldots \eta_N$ are linear forms of the basis elements $\l_1,\ldots \l_n$ of $\liet^*$. Let $\coeff(\eta_i,j)$ denote the coefficient of $\l_j$ in $\eta_i$ ($1\le i\le N, 1\le j \le n$) and introduce the notation
\[\deg(\eta_1,\ldots, \eta_N;m)=\#\{i;\;\coeff(\eta_i,m)\neq 0\}\}.\]
Let $\Sigma \subset W$ be a $T$-invariant subvariety. It is clear from the formula \eqref{mdegformula} that the $\l_m$-degree of $\epd{\Sigma, W}$ satisfies 
\begin{equation}\label{degree}
\deg_{\l_m}\epd{\Sigma,W} \le \deg(\eta_1,\ldots, \eta_N;m)
\end{equation}
for any $1\le m \le n$. 

Finally we state one of the basic properties listed in \cite{bsz} Proposition 2.3 as a lemma here as this will be used repeatedly later. 
\begin{lemma}[Elimination property, \cite{bsz} Prop 2.3] Let $\Sigma\subset W$ be a closed $T$-invariant
  subvariety and denote by $I(\Sigma)$ the ideal of functions
  vanishing on $\Sigma$. Fix a polynomial $f\in \CC[y_1,\dots,y_N]$ of
  weight $\eta_0$, and let $\Sigma_f$ be the variety in
  $W\oplus\CC y_0$ with ideal generated by $I(\Sigma)$ and
  $y_0-f$. Then
\[\epd{\Sigma_f,W\oplus\CC y_0}=\eta_0\cdot \epd{\Sigma,W}\]
\end{lemma}

\begin{example} Let $W=\CC^4$ endowed with a
$T=(\CC^*)^3$-action, whose weights $\eta_1, \eta_2,\eta_3$ and
$\eta_4$ span $\mathfrak{t}^*$, and satisfy
$\eta_1+\eta_3=\eta_2+\eta_4$. Choose $p=(1,1,1,1)\in W$; then the affine toric variety
\[\overline{T\cdot p}=\{(y_1,y_2,y_3,y_4)\in\CC^4;\; y_1y_3=y_2y_4\}.\]
is a hypersurface and its equivariant dual is given by the weight of the equation: 
\[\epd{\overline{T\cdot p},W} = \eta_1+\eta_3=\eta_2+\eta_4.\]
An other way to see this is to fix the monomial order $>$ induced from $y_1>y_2>y_3>y_4$, then the ideal $I=(y_1y_3-y_2y_4)$ has initial ideal $in_I=(y_1y_3)$ whose spectrum is the union of the hyperplanes $\{y_1=0\}$ and $\{y_3=0\}$ with duals $\eta_1$, $\eta_3$ respectively. 
\end{example}

\begin{remark}\label{remark:topdef}
An alternative and slightly more general topological definition of the equivariant dual is the following, see \cite{fultonnotes,kaz97,eg97} for details. For a Lie group $G$ let $EG\to BG$ be a right principal $G$-bundle with $EG$ contractible. Such a bundle is universal in the topological setting: if $E\to B$ is any principal $G$-bundle, then there is a map $B \to BG$, unique up to homotopy, such that $E$ is isomorphic to the pullback of $EG$. If $X$ is a smooth algebraic $G$-variety then the topological definition of the $G$-equivariant cohomology of $X$ is 
\[H_G^*(X)=H^*(EG \times_G X).\]
If $Y$ is a $G$-invariant subvariety then $Y$ represents a $G$-equivariant cohomology class in the equivariant cohomology of $X$, namely the ordinary Poincar\'e dual of $EG \times_G Y$ in $EG \times_G X$. This is the equivariant dual of $Y$ in $X$:
\[\epd{Y,X}=\mathrm{PD}(EG \times_G Y, EG \times_G X).\]
\end{remark}

\subsection{The Rossman formula} \label{subsec:rossman} Let $Z$ be a complex manifold with a holomorphic $T$-action, and let
$M\subset Z$ be a $T$-invariant analytic subvariety with an isolated
fixed point $p\in M^T$. Then one can find local analytic coordinates
near $p$, in which the action is linear and diagonal. Using these
coordinates, one can identify a neighborhood of the origin in $\TT_pZ$
with a neighborhood of $p$ in $Z$. We denote by $\tc_pM$ the part of
$\TT_pZ$ which corresponds to $M$ under this identification;
informally, we will call $\tc_pM$ the $T$-invariant {\em tangent cone}
of $M$ at $p$. This tangent cone is not quite canonical: it depends on
the choice of coordinates; the equivariant dual of
$\Sigma=\tc_pM$ in $W=\TT_pZ$, however, does not. Rossmann named this
 the {\em equivariant multiplicity of $M$ in $Z$ at $p$}:
\begin{equation}\label{emult}
   \emu_p[M,Z] \overset{\mathrm{def}}= \epd{\tc_pM,\TT_pZ}.
\end{equation}

\begin{remark}
In the algebraic framework one might need to pass to the {\em
tangent scheme} of $M$ at $p$ (cf. \cite{fulton}). This is canonically
defined, but we will not use this notion.
\end{remark}
The analog of the Atiyah-Bott formula for singular subvarieties of smooth ambient manifolds is the following
\begin{proposition}[Rossmann's localisation formula \cite{rossmann}]\label{rossman} Let $\mu \in H_T^*(Z)$ be an equivariant class represented by a holomorphic equivariant map $\mathfrak{t} \to\Omega^\bullet(Z)$. Then 
\begin{equation}
  \label{rossform}
  \int_M\mu=\sum_{p\in M^T}\frac{\emu_p[M,Z]}{\mathrm{Euler}^T(\TT_pZ)}\cdot\mu^{[0]}(p),
\end{equation}
where $\mu^{[0]}(p)$ is the differential-form-degree-zero component
of $\mu$ evaluated at $p$.  
\end{proposition}

\subsection{Equivariant localisation on $\widetilde{CX}^{[k+1]}$ for $k\le n$}\label{subsec:loc}
In this subsection we start to develop a two step equivariant localisation method on $\widetilde{CX}^{[k+1]}$ using the Rossmann formula. As the partial resolution $\widetilde{CX}^{[k+1]}$ described in \S\ref{subsec:blowingup} is defined only for $k\le n$ we impose this condition in this section. 

Recall from $\S$\ref{subsec:blowingup} the blow-up definition  
\[\widetilde{CX}^{[k+1]}_p =\GL(n) \times_{P_{k,n}}  \overline{P_{k,n}\cdot \mathfrak{p}_{k,n}} \to \overline{CX}^{[k+1]}_p\]
sitting in $\flag_k(\symdot)$ which fibres over the flag manifold $\GL(n)/P_{k,n}=\flag_k(\CC^n)$:
\begin{equation*}
\xymatrix{\widetilde{CX}^{[k+1]}_p \ar[r]^-{\rho} \ar[d]^{\mu} & \overline{CX}^{[k+1]}_p \subset \flag_k(\symdot) \\
\Hom(\CC^k,\CC^n)/B_k=\flag_k(\CC^n)& } 
\end{equation*}
and the fibres of $\mu$ are isomorphic to $\overline{P_{k,n}\cdot p_k}\subset \flag_k(\symdot)$.
The corresponding fibred version of this diagram over $X$ gives the partial resolution of the curvilinear Hilbert scheme $\widetilde{CX}^{[k+1]} \to \overline{CX}^{[k+1]}$:
\begin{equation}\label{diagram}
\xymatrix{\widetilde{CX}^{[k+1]} \ar[r]^-{\rho} \ar[d]^{\mu} & \overline{CX}^{[k+1]}\subset \flag_k(\symdotx) \\
\flag_k(T_X^*) \ar[d]^{\tau} & \\
X & } 
\end{equation}
where $\flag_k(T_X^*)$ is the flag bundle of the cotangent bundle $T_X^*$, and over every point $p\in X$ we get back the previous diagram, that is, the fibres of $\tilde{\pi}=\tau \circ \mu: \widetilde{CX}^{[k+1]} \to X$ are canonically isomorphic to $\widetilde{CX}^{[k+1]}_p$. 

Fix a Chern polynomial $P=P(c_1,\ldots, c_{r(k+1)})$ of degree $\dim \overline{CX}^{[k+1]}=n+(n-1)k$ where $c_i=c_i(F^{[k+1]})$ are the Chern classes of the tautological rank $r(k+1)$ bundle on the curvilinear Hilbert scheme.  
To evaluate the integral $\int_{\widetilde{CX}^{[k+1]}}P$ we can first integrate (push forward) along the fibres of $\tilde{\pi}:\widetilde{CX}^{[k+1]} \to X$ followed by integration over $X$. These fibres are canonically isomorphic to $\widetilde{CX}^{[k+1]}_p$ endowed with a natural $\GL(n)$ action induced by the standard $\GL(n)$ action on $\CC^n$ and we can use this action to perform torus equivariant localisation on $\widetilde{CX}^{[k+1]}_p$ to integrate along the fibres. Recall that $K=U(1)^n$ is the maximal compact subgroup of the maximal complex torus $T$ of $\GL(n,\CC)$ and $\mathfrak{t}=\mathrm{Lie}(K)$. Take a fibrewise equivariant extension 
\[\alpha \in H_{T}^f=S^\bullet \mathfrak{t}^*\otimes (\Omega^\bullet(\widetilde{CX_p}^{[k+1]})^{\kt} \otimes \Omega^\bullet(X)\]
with respect to the torus action on $\widetilde{CX_p}^{[k+1]}$. Then $\alpha$ is a polynomial function on $\mathfrak{t}$ with values in the $\Omega^\bullet(X)$-module $\Omega^\bullet(\widetilde{CX_p}^{[k+1]})^{K} \otimes \Omega^\bullet(X)$.
Integration along the fibre is the map  
\[H_T^f \to S^\bullet \mathfrak{t}^* \otimes \Omega^\bullet(X)\] 
defined as
\[(\int \alpha)(X)=\int_{\widetilde{CX}^{[k+1]}_p} \alpha^{[\mathrm{dim}(\widetilde{CX}^{[k+1]}_p)]}(X) \text{ for all } X\in \mathfrak{t}\]
where $\alpha^{[\mathrm{dim}(\widetilde{CX}^{[k+1]}_p)]}$ is the $(\Omega^\bullet(\widetilde{CX_p}^{[k+1]})^{K}$-degree-$d$ part of $\alpha$ with $d=\dim(\widetilde{CX}^{[k+1]}_p)$.  In short, we consider the $\Omega^\bullet(X)$ part of $\alpha$ as a constant and apply the standard localisation map \eqref{localisationmap} on $S^\bullet \mathfrak{t}^*\otimes (\Omega^\bullet(\widetilde{CX_p}^{[k+1]})^{K}$.  

Note that $\mu:\widetilde{CX}^{[k+1]}_p \to \flag_k(\CC^n)$ gives a $\GL(n)$-equivariant fibration over the flag manifold $\flag_k(\CC^n)$. Let $e_1,\ldots, e_n \in \CC^n$ be an eigenbasis of $\CC^n$ for the $T$ action on $\widetilde{CX}^{[k+1]}_p$ with weights $\l_1,\ldots, \l_n\in \mathfrak{t}^*$ and let 
\[\ff=(\langle e_1 \rangle \subset \langle e_1,e_2 \rangle \subset \ldots \subset \langle e_1,\ldots,e_k \rangle \subset \CC^n)\]
denote the standard flag in $\CC^n$ fixed by the parabolic $P_{k,n} \subset \GL(n)$.
Since the torus action on $\widetilde{CX}^{[k+1]}_p$ is obtained by the restriction of a $\GL(n)$-action to
its subgroup of diagonal matrices $T_n$, the Weyl group of
permutation matrices $S_n$ acts transitively on the fixed points set $\flag_k(\CC^n)^{T_n}$ taking the standard flag $\ff$ to $\sigma(\ff)$ and Proposition \ref{abbv} gives us   
\begin{equation} \label{flagloc}
\int_{\widetilde{CX}^{[k+1]}_p}\alpha= \sum_{\sigma\in\sg n/\sg{n-k}}
\frac{\alpha_{\sigma(\ff)}}{\prod_{1\leq m \leq
k}\prod_{i=m+1}^n(\lambda_{\sigma\cdot
    i}-\lambda_{\sigma\cdot m})},
\end{equation}
where 
\begin{itemize}
\item $\sigma$ runs over the ordered $k$-element subsets of $\{1,\ldots, n\}$ labeling the fixed flags $\sigma(\ff)=(\langle e_{\sigma(1)} \rangle \subset \ldots \subset \langle e_{\sigma(1)},\ldots, e_{\sigma(k)} \rangle \subset \CC^n)$ in $\CC^n$, 
\item $\prod_{1\leq m \leq k}\prod_{i=m+1}^n(\lambda_{\sigma(i)}-\lambda_{\sigma(m)})$ is the equivariant Euler class of the tangent space of $\flag_k(\CC^n)$ at $\s(\ff)$,
\item if $\widetilde{CX}^{[k+1]}_{\sigma(\ff)}=\mu^{-1}(\sigma(\ff))$ denotes the fibre then $\alpha_{\sigma(\ff)}=(\int_{\widetilde{CX}^{[k+1]}_{\sigma(\ff)}} \alpha)^{[0]}(\sigma(\ff))\in S^\bullet \mathfrak{t}^* \otimes \Omega^\bullet(X)$ is the differential-form-degree-zero part evaluated at $\sigma(\ff)$ and $\alpha_{\sigma(\ff)}=\sigma \cdot \alpha_\ff$ with respect to the natural Weyl group action on $S^\bullet \mathfrak{t}^*$. 
\end{itemize}
 In particular, when $\alpha=\alpha(\theta_1,\ldots, \theta_r,\eta_1,\ldots, \eta_k)$ is a bi-symmetric polynomial in the Chern roots $\theta_i$ of the pull-back of $F$ over $\widetilde{CX}^{[k+1]}_p \subset \flag_k(\symdot)$ and the Chern roots $\eta_j$ of the tautological rank $k$ bundle $\cale$ then $\alpha_{\ff}$ is a polynomial in two sets of variables: in the basic weights 
$\lambda=(\lambda_1\ldots \lambda_n)$ and in the 
$\theta=(\theta_1\ldots \theta_r)$. Since $\mu^{-1}(\ff)$ is invariant under
$P_{k,n}$ only, this polynomial is not necessarily symmetric in the
$\lambda$'s. Note that $\alpha_\ff$ contains only Chern roots of the tautological rank $k$ bundle $\cale$ and therefore it does not depend on
  the last $n-k$ basic weights: $\lambda_{k+1},\dots,\lambda_n \in \mathfrak{t}^*$.
\[\alpha_{\ff}=\alpha_{\ff}(\theta_1,\ldots, \theta_r,\lambda_1,\ldots, \lambda_k)\]
and 
\begin{equation}\label{alphasigmaf}
\alpha_{\s(\ff)}=\sigma \cdot \alpha_\ff=\alpha_\ff(\theta_1,\ldots, \theta_r,\l_{\s(1)},\ldots ,\l_{\s(k})\in S^\bullet \mathfrak{g}^* \otimes H^*(X)
\end{equation}
is the $\sigma$-shift of the polynomial $\alpha_{\ff}$ corresponding to the distinguished fixed flag $\ff$.

\subsection{Transforming the localisation formula into iterated residue}\label{subsec:transform}

In this section we transform the right hand side of \eqref{flagloc} into an iterated residue. This step turns out to be crucial in handling the combinatorial complexity of the Atiyah-Bott localisation formula and captures the symmetry of the fixed point data in an efficient way which enables us to prove the vanishing of the contribution of all but one of the fixed points. 

To describe this formula, we will need the notion of an {\em iterated
  residue} (cf. e.g. \cite{szenes}) at infinity.  Let
$\omega_1,\dots,\omega_N$ be affine linear forms on $\CC^k$; denoting
the coordinates by $z_1,\ldots, z_k$, this means that we can write
$\omega_i=a_i^0+a_i^1z_1+\ldots + a_i^kz_k$. We will use the shorthand
$h(\bz)$ for a function $h(z_1\ldots z_k)$, and $\dbz$ for the
holomorphic $n$-form $dz_1\wedge\dots\wedge dz_k$. Now, let $h(\bz)$
be an entire function, and define the {\em iterated residue at infinity}
as follows:
\begin{equation}
  \label{defresinf}
 \ires \frac{h(\bz)\,\dbz}{\prod_{i=1}^N\omega_i}
  \overset{\mathrm{def}}=\left(\frac1{2\pi i}\right)^k
\int_{|z_1|=R_1}\ldots
\int_{|z_k|=R_k}\frac{h(\bz)\,\dbz}{\prod_{i=1}^N\omega_i},
 \end{equation}
 where $1\ll R_1 \ll \ldots \ll R_k$. The torus $\{|z_m|=R_m;\;m=1 \ldots
 k\}$ is oriented in such a way that $\res_{z_1=\infty}\ldots
 \res_{z_k=\infty}\dbz/(z_1\cdots z_k)=(-1)^k$.
We will also use the following simplified notation: $\sires \overset{\mathrm{def}}=\ires.$

In practice, one way to compute the iterated residue \eqref{defresinf} is the following algorithm: for each $i$, use the expansion
 \begin{equation}
   \label{omegaexp}
 \frac1{\omega_i}=\sum_{j=0}^\infty(-1)^j\frac{(a^{0}_i+a^1_iz_1+\ldots
   +a_{i}^{q(i)-1}z_{q(i)-1})^j}{(a_i^{q(i)}z_{q(i)})^{j+1}},
   \end{equation}
   where $q(i)$ is the largest value of $m$ for which $a_i^m\neq0$,
   then multiply the product of these expressions with $(-1)^kh(z_1\ldots
   z_k)$, and then take the coefficient of $z_1^{-1} \ldots z_k^{-1}$
   in the resulting Laurent series.

We repeat the proof of the following iterated residue theorem from \cite{bsz}.
\begin{proposition}[(\cite{bsz} Proposition 5.4)]\label{abtoresidue} For any homogeneous polynomial $Q(\bz)$ on $\CC^k$ we have
\begin{equation}\label{flagres}
\sum_{\sigma\in\sg n/\sg{n-k}}
\frac{Q(\lambda_{\sigma(1)},\ldots ,\lambda_{\sigma(k)})}
{\prod_{1\leq m\leq k}\prod_{i=m+1}^n(\lambda_{\sigma\cdot
    i}-\lambda_{\sigma\cdot m})}=\sires
\frac{\prod_{1\leq m<l\leq k}(z_m-z_l)\,Q(\bz)\dbz}
{\prod_{l=1}^k\prod_{i=1}^n(\lambda_i-z_l)}
\end{equation}
\end{proposition}

\begin{proof}
  We compute the iterated residue \eqref{flagres} using the Residue
  Theorem on the projective line $\CC\cup\{\infty\}$.  The first
  residue, which is taken with respect to $z_k$, is a contour
  integral, whose value is minus the sum of the $z_k$-residues of the
  form in \eqref{flagres}. These poles are at $z_k=\lambda_j$,
  $j=1\ldots n$, and after canceling the signs that arise, we obtain the
  following expression for the right hand side of  \eqref{flagres}:
\[
\sum_{j=1}^n \frac{\prod_{1\leq m<l\leq
    k-1}(z_m-z_l)\,\prod_{l=1}^{k-1}(z_l-\lambda_j)\,Q(z_1\ldots
  z_{k-1},\lambda_j)\;dz_1\dots dz_{k-1}}
{\prod_{l=1}^{k-1}\prod_{i=1}^n(\lambda_i-z_l)\prod_{i\neq
    j}^n(\lambda_i-\lambda_j)}.
\]
After cancellation and exchanging the sum and the residue operation,
at the next step, we have
\[
(-1)^{k-1}\sum_{j=1}^n\res_{z_{k-1}=\infty} \frac{\prod_{1\leq
m<l\leq
    k-1}(z_m-z_l)\,Q(z_1\ldots z_{k-1},\lambda_j)\;dz_1\dots
  dz_{k-1}} {\prod_{i\neq j}^n
  \left((\lambda_i-\lambda_j)\prod_{l=1}^{k-1}(\lambda_i-z_l)\right)}.
\]
Now we again apply the Residue Theorem, with the only difference that
now the pole $z_{k-1}=\lambda_j$ has been eliminated. As a result,
after converting the second residue to a sum, we obtain
\[
(-1)^{2k-3}\sum_{j=1}^n\sum_{s=1,\,s\neq j}^n \frac{\prod_{1\leq
m<l\leq
    k-2}(z_l-z_m)\,Q(z_1\ldots
  z_{k-2},\lambda_s,\lambda_j)\;dz_1\dots dz_{k-2}}
{(\lambda_s-\lambda_j)\prod_{i\neq j,s}^n
  \left((\lambda_i-\lambda_j)(\lambda_i-\lambda_s)\prod_{l=1}^{k-1}(\lambda_i-z_l)\right)}.
\]
Iterating this process, we arrive at a sum very similar to
(\ref{flagloc}). The difference between the two sums will be the
sign: $(-1)^{k(k-1)/2}$, and that the $k(k-1)/2$ factors of the form
$(\lambda_{\sigma(i)}-\lambda_{\sigma(m)})$ with $1\le m<i\le k$ in
the denominator will have opposite signs. These two differences
cancel each other, and this completes the proof.
\end{proof}
\begin{remark}
  Changing the order of the variables in iterated residues, usually,
  changes the result. In this case, however, because all the poles are
  normal crossing, formula \eqref{flagres} remains true no matter in
  what order we take the iterated residues.
\end{remark}

Proposition \ref{abtoresidue}  together with \eqref{flagloc} and \eqref{alphasigmaf} gives
\begin{proposition}\label{propflag} Let $k\le n$ and $\alpha=\alpha(\theta_1,\ldots, \theta_r,\eta_1,\ldots, \eta_k)$ be a bi-symmetric polynomial in the Chern roots $\theta_i$ of the pull-back of $F$ over $\widetilde{CX}^{[k+1]}_p \subset \flag_k(\symdot)$ and the Chern roots $\eta_j$ of the tautological rank $k$ bundle $\cale$. Then 
\begin{equation*}
\int_{\widetilde{CX}^{[k+1]}_p}\alpha=\sires
\frac{\prod_{1\leq m<l\leq k}(z_m-z_l)\,\alpha_{\ff}(\theta_1,\ldots, \theta_r,z_1,\ldots, z_k)\dbz}
{\prod_{l=1}^k\prod_{i=1}^n(\lambda_i-z_l)}
\end{equation*}
where $s_i(\bz)=s_i(z_1,\ldots, z_k)$ denotes the $i$th symmetric polynomial in $z_1,\ldots, z_k$.
\end{proposition}
Next, we proceed a second localisation on the fibre 
\[\widetilde{CX}^{[k+1]}_{\ff}=\mu^{-1}(\ff)\simeq \overline{P_{k,n}\cdot p_k}\subset \flag_k(\symdot)\]
 to compute $\alpha_\ff(\mathbf{\theta},\bz)$. Since $\widetilde{CX}^{[k+1]}_{\ff}$
 is invariant under the $T$-action on $\flag_k(\symdot)$, we can apply Rossmann's integration formula, see Proposition \ref{rossman}. Note that the fibre $\widetilde{CX}^{[k+1]}_{\ff}=\overline{P_{k,n}\cdot p_k}$ sits in the submanifold 
\[\flag_k^*(\symdot)=\{V_1 \subset \ldots \subset V_k\subset \symdot: \dim(V_i)=i, V_i\subset \mathrm{Span}_\CC(e_\tau:\Sigma \tau \le i)\}\]
of $\flag_k(\symdot)$. Since the subspaces 
\[W_i=\mathrm{Span}_\CC(e_\tau:\Sigma \tau \le i)\subset \symdot\]
are invariant under the upper Borel $B_n \subset \GL(n)$ which fixes the flag $\ff$,
\[\flag_k^*(\symdot) \subset \flag_k(\symdot)\]
is a $B_n$-invariant subvariety. 

We apply the Rossman formula for $M=X_\ff, Z=\flag_k^*(\symdot)$ and $\mu=\alpha_\ff$. The fixed points on 
\[Z=\flag_k^*(\symdot) \subset \bigoplus_{i=1}^k W_1 \wedge \ldots \wedge W_i\]  
are parametrised by {\it admissible} sequences of partitions $\bipi=(\pi_1,\ldots, \pi_k)$. We call a sequence of partitions $\bipi=(\pi_1 \ldots \pi_k)\in\Pi^{\times d}$ admissible if
\begin{enumerate}
\item $\Sigma \pi_l\le l$ for $1\le l \le k$, and 
\item $\pi_l\neq\pi_m$ for $1\leq l\neq m\leq k$. 
\end{enumerate}
We will denote the set of admissible sequences of length $k$ by $\Bipi$. The corresponding fixed point is then 
\[\bigoplus_{i=1}^k e_{\pi_1} \wedge \ldots \wedge e_{pi_i} \in \bigoplus_{i=1}^k W_1 \wedge \ldots \wedge W_i\]
where $e_{\pi}=\prod_{j\in \pi}e \in \sym^{|\pi|}\CC^n$.

Then the Rossman formula \eqref{rossform} and Proposition \ref{propflag} give us 
\begin{proposition}\label{propint} Let $k\le n$ and let $\alpha=\alpha(\theta_1,\ldots, \theta_r,\eta_1,\ldots, \eta_k)$ be a bi-symmetric polynomial in the Chern roots $\theta_i$ of the pull-back of $F$ over $\widetilde{CX}^{[k+1]}_p \subset \flag_k(\symdot)$ and the Chern roots $\eta_j$ of the tautological rank $k$ bundle $\cale$. Then 
\begin{equation}\label{intnumberone} 
\int_{\widetilde{CX}^{[k+1]}_p}\alpha=\sum_{\bipi\in\Bipi \cap \overline{P_{k,n} \cdot p_k}} \sires \frac{
Q_\bipi(\bz)\,\prod_{m<l}(z_m-z_l) \alpha(\mathbf{\theta},z_{\pi_1},\ldots, z_{\pi_k})}{
\prod_{l=1}^k\prod_{\tau\leq l}^{\tau\neq\pi_1\ldots \pi_l}
(z_{\tau}-z_{\pi_l})  \prod_{l=1}^k\prod_{i=1}^n(\lambda_i-z_l)} \,\dbz.
  \end{equation}
where $Q_\bipi(\bz)=\emu_\bipi[X_\ff,\flag_k^*]$ and $z_\pi=\sum_{i\in \pi}z_i$.
\end{proposition}  

This formula reduces the computation of the tautological integrals $\int_{\widetilde{CX}^{[k+1]}_p}\alpha$ to determine the fixed point set $\Bipi \cap \widetilde{CX}^{[k+1]}_\ff$ and the multidegree $Q_\bipi(\bz)=\emu_\bipi[X_\ff,\flag_k^*]$ of the tangent cone of $\widetilde{CX}^{[k+1]}_\ff$ in $\flag_k^*(\symdot)$. 

%If $\alpha$ represents the cohomology class $R(u,\tilde{\pi}^*h)\in H^*(\tilde{\calx}_k)$, where $R$ is %a homogeneous polynomial of degree $\mathrm{dim}\tilde{\calx}_k=n+k(n-1)$, then integrating on %$X$ and using the fact that $\tilde{pi}_*\tilde{\pi}^*h=h$ results 
%\begin{equation} \label{intnumberone}
%\int_{\tilde{\calx}_k}\alpha=\int_X \sum_{\bipi\in\Bipi} \sires \frac{
%Q_\bipi(\bz)\,\prod_{m<l}(z_m-z_l) R(z_{\pi_1}+\ldots +z_{\pi_k},h)}{
%\prod_{l=1}^k\prod_{\tau\leq l}^{\tau\neq\pi_1\ldots \pi_l}
%(z_{\tau}-z_{\pi_l})  \prod_{l=1}^k\prod_{i=1}^n(\lambda_i-z_l)} \,\dbz, 
%  \end{equation}
%where integration over $X$ on the right hand side means the substitution $h^n=d$.

\section{The residue vanishing theorem}

The first immediate problem arising with our formula \eqref{intnumberone} is that we do not have a complete description of the fixed point set $\Bipi \cap \widetilde{CX}^{[k+1]}_\ff$ and in fact it seems to be a hard question to decide which torus fixed points on $\flag_k^*(\symdot)$ sit in the orbit closure $\widetilde{CX}^{[k+1]}_\ff=\overline{P_{k,n}\cdot p_k}$.
The second problem we face is how to compute the multidegrees $Q_\bipi(\bz)=\emu_\bipi[X_\ff,\flag_k^*]$ for those admissible sequences which represent fixed points in $\overline{P_{k,n}\cdot p_k}$. We postpone this second problem to the next section and here we focus on the first question which has a particularly nice--and surprising--answer. Namely, we do not need to know which fixed points sit in $\overline{P_{k,n}\cdot p_k}$ because our limited knowledge on the equations of the $P_{k,n}$-orbit is enough to show that all but one terms on the right hand side of \eqref{intnumberone} vanish. This key feature of the iterated residue has already appeared in\cite{bsz} but here we need to prove a stronger version where the total degree of the rational forms are zero. 
We devote the rest of this section to the proof of  
\begin{theorem}[\textbf{The Residue Vanishing Theorem}]\label{vanishtheorem} Let $k+1\le n$ and $\alpha=\alpha(\theta_1,\ldots, \theta_r,\eta_1,\ldots, \eta_k)$ be a bi-symmetric polynomial in the Chern roots $\theta_i$ of the pull-back of $F$ over $\widetilde{CX}^{[k+1]}_p \subset \flag_k(\symdot)$ and the Chern roots $\eta_j$ of the tautological rank $k$ bundle $\cale$. Then 
\begin{enumerate}
\item All terms but the one corresponding to $\bipi_\dist=([1],[2],\ldots, [k])$ vanish in \eqref{intnumberone} leaving us with
\begin{equation} \label{intnumberoneandhalf}
\int_{\widetilde{CX}^{[k+1]}_p}\alpha=\sires \frac{
Q_{[1],\ldots, [k]}(\bz)\,\prod_{m<l}(z_m-z_l) \alpha(\mathbf{\theta},\bz)}{
\prod_{\um \tau \le l \le k}
(z_\tau-z_l)  \prod_{l=1}^k\prod_{i=1}^n(\lambda_i-z_l)} \,\dbz.
  \end{equation} 
 \item If $|\tau|\ge 3$ then $Q_k(\bz)=Q_{([1],\ldots, [k])}(\bz)$ is divisible by $z_\tau-z_l$ for all $l \ge \um \tau$ implying the simplified formula
\begin{equation} \label{intnumberthree}
\int_{\widetilde{CX}^{[k+1]}_p}\alpha=\sires \frac{
Q_k(\bz)\,\prod_{m<l}(z_m-z_l) \alpha(\theta,\bz)}{
\prod_{m+r \le l \le k}
(z_m+z_r-z_l)  \prod_{l=1}^k\prod_{i=1}^n(\lambda_i-z_l)} \,\dbz.
  \end{equation}
\end{enumerate}
\end{theorem} 

\begin{remark}\label{remarkq}
\begin{enumerate}
\item 
The geometric meaning of $Q_k(\bz)$ in \eqref{intnumberthree} is the following, see also \cite{bsz} Theorem 6.16. Let $T_k\subset B_k\subset \GL(k)$ be the subgroups of invertible
  diagonal and upper-triangular matrices, respectively; denote the
  diagonal weights of $T_k$ by $z_1, \ldots, z_k$.  Consider the $\GL(k)$-module of 3-tensors $\Hom(\CC^k,\sym^2\CC^k)$; identifying the
  weight-$(z_m+z_r-z_l)$ symbols $q^{mr}_l$ and $q^{rm}_l$, we can
  write a basis for this space as follows:
\[ \Hom(\CC^k,\sym^2\CC^k)=\bigoplus \CC q^{mr}_l,\;  1\leq m,r,l \leq k.
\]
Consider the point $\epsilon=\sum_{m=1}^k\sum_{r=1}^{k-m}q_{mr}^{m+r}$
in  the $B_k$-invariant subspace
\begin{equation*}
  \label{nhmodule}
    N_k = \bigoplus_{1\leq m+r\leq l\leq k} \CC q^{mr}_l\subset
\Hom(\CC^k,\sym^2\CC^k).
\end{equation*}
Set the notation $\OO_k$ for the orbit closure
$\overline{B_k\epsilon}\subset N_k$, then $Q_k(\bz)$ is the $T_k$-equivariant
Poincar\'e dual $Q_k(\bz) = \epd{\OO_k,N_k}_{T_k}$,
which is a homogeneous polynomial of degree
$\dim(N_k)-\dim(\OO_k).$. For small $k$ these polynomials are the following (see \cite{bsz} $\S7$):
\[Q_2=Q_3=1, Q_4=2z_1+z_2-z_4\]
\[Q_5=(2z_1+z_2-z_5)(2z_1^2 +3z_1z_2-2z_1z_5+2z_2z_3-z_2z_4-z_2z_5-z_3z_4+z_4z_5).\] 
\item To understand the significance of this vanishing theorem we note that while the fixed point set $\Bipi$ on $\flag_k^*(\symdot)$ is well understood, it is not clear which of these fixed points sit in $X_{\mathbf{f}}$. But we have enough information to prove that none of those fixed points in $X_{\mathbf{f}}$ contribute to the iterated residue except for the distinguised fixed point $\bipi_\dist=([1],[2],\ldots, [k])$.
\item The Residue Vanishing Theorem is valid under the condition $k+1\le n$ which is slightly stronger than the condition $k\le n$ we worked with so far and which guaranteed the existence of $\widetilde{CX}^{[k+1]}_p$. We will remedy this condition in \S\ref{sec:extendtherange}.
\end{enumerate}
\end{remark}
\begin{remark}\label{remark:resolution}
Remark \ref{remark:pullback} for singular varieties and ordinary compactly supported differential forms holds for compactly supported equivariant forms as follows. Let $T$ be a complex torus and $f:M\to N$ be a 
smooth proper $T$-equivariant map between smooth quasiprojective
varieties. Now assume that $X\subset M$ and $Y\subset N$ are possibly
singular $T$-invariant closed subvarieties, such that $f$ restricted
to $X$ is a birational map from $X$ to $Y$. Next, let $\mu$ be an
equivariantly closed differential form on $N$ with values in
polynomials on $\mathfrak{t}$. Then the integral of $\mu$ on the smooth part of
$Y$ is absolutely convergent; we denote this by $\int_Y\mu$. With this
convention we again have
\begin{equation}
  \label{degsmooth}
\int_X   f^*\mu= \int_Y\mu,
\end{equation}
and we can define integrals of equivariant forms on singular quasi-projective varieties simply by passing to any partial equivariant resolution or equivalently to integration over the smooth locus. In particular, applying this for the partial resolution $\rho: \widetilde{CX}^{[k+1]}_p \to \overline{CX}^{[k+1]}_p$ we get 
\[\int_{\overline{CX}^{[k+1]}_p}\alpha=\int_{\widetilde{CX}^{[k+1]}_p}\rho^* \alpha\]
for any $\alpha \in \Omega^*(\overline{CX}^{[k+1]}_p)$ closed compactly supported differential form.
\end{remark}

\subsection{The vanishing of residues}
\label{subsec:vanres}

In this subsection following \cite{bsz} $\S6.2$ we describe the conditions under which iterated
residues  of the type appearing in the sum in \eqref{intnumberone}
vanish and we prove Theorem \ref{vanishtheorem}.

We start with the 1-dimensional case, where the residue at infinity
is defined by \eqref{defresinf} with $d=1$. By bounding the integral
representation along a contour $|z|=R$ with $R$ large, one can
easily prove
\begin{lemma}\label{1lemma}
  Let $p(z),q(z)$ be polynomials of one variable. Then
\[\res_{z=\infty}\frac{p(z)\,dz}{q(z)}=0\quad\text{if }\deg(p(z))+1<\deg(q).
\]
\end{lemma}
Consider now the multidimensional situation. Let $p(\bz),q(\bz)$ be
polynomials in the $k$ variables $z_1\ldots z_k$, and assume that
$q(\bz)$ is the product of linear factors $q=\prod_{i=1}^N L_i$, as
in \eqref{intnumberone}. We continue to use the notation $\dbz=dz_1\dots dz_k$.
We would like to formulate conditions under which the iterated
residue
\begin{equation}
  \label{ires}
\ires\frac{p(\bz)\,\dbz}{q(\bz)}
\end{equation}
vanishes. Introduce the following notation:
\begin{itemize}\label{notations}
\item For a set of indices $S\subset\{1\ldots k\}$, denote by $\deg(p(\bz);S)$
  the degree of the one-variable polynomial $p_S(t)$ obtained from $p$
  via the substitution $z_m\to
  \begin{cases}
t\text{ if }m\in S,\\ 1\text{ if }m\notin S.
  \end{cases}$. When $p(\bz)$ is the product of linear forms, $\deg(p(\bz);S)$ is the number of terms with nonzero coefficients in front of at least one of $z_s$ for $s\in S$.
\item For a nonzero linear form $L=a_0+a_1z_1+\ldots +a_kz_k$,
  denote by $\coeff(L,z_l)=a_i$ the coefficient in front of $z_i$;
\item finally, for $1\leq m\leq k$, set
  \[\lead(q(\bz);m)=\#\{i;\;\max\{l;\;\coeff(L_i,z_l)\neq0\}=m\},\]
which is the number of those factors $L_i$ in which the coefficient
of $z_m$ does not vanish, but the coefficients of $z_{m+1},\ldots,
z_k$ are $0$.
\end{itemize}
We can group the $N$ linear factors of $q(\bz)$ according to the
nonvanishing coefficient with the largest index; in particular, for
$1\leq m\leq k$ we have
\[   \deg(q(\bz);m)\geq\lead(q(\bz);m),\, \text{ and } \sum_{m=1}^k\lead(q(\bz);m)=N.
\]
 \begin{proposition}[\cite{bsz} Proposition 6.3]
  \label{vanishprop}
Let $p(\bz)$ and $q(\bz)$ be polynomials in the variables $z_1\ldots 
z_k$, and assume that $q(\bz)$ is a product of linear factors:
$q(\bz)=\prod_{i=1}^NL_i$; set $\dbz=dz_1\dots dz_k$. Then
\[ \ires\frac{p(\bz)\dbz}{q(\bz)} = 0
\]
if for some $l\leq k$, either of the following two options hold:
\begin{itemize}
\item $\deg(p(\bz);k,k-1,\dots,l)+k-l+1<\deg(q(\bz);k,k-1,\dots,l),$
\\ or
\item  $\deg(p(\bz);l)+1<\deg(q(\bz);l)=\lead(q(\bz);l)$.
\end{itemize}
\end{proposition}
Note that for the second option, the equality
$\deg(q(\bz);l)=\lead(q(\bz);l)$ means that
  \begin{equation}
    \label{op2cond}
\text{ for each }i=1\ldots N\text{ and }m>l,\,
\coeff(L_i,z_{l})\neq0\text{ implies }\coeff(L_i,z_{m})=0.
  \end{equation}

We are ready to proof the Residue Vanishing Theorem. Recall that our goal is to show that all the terms of the sum in \eqref{intnumberone} vanish except for the one corresponding to
$\bipi_\dist=([1]\ldots  [k])$. The plan is to apply Proposition \ref{vanishprop} in stages to show that the itrated residue vanishes unless $z_i=[i]$ holds, starting with $i=k$ and going backwards. 

Fix a sequence $\bipi=(\pi_1,\dots,\pi_k)\in\Bipi$, and consider the
iterated residue corresponding to it on the right hand side of
\eqref{intnumberone}. The expression under the residue is the product of
two fractions:
\[\frac{p(\bz)}{q(\bz)}=\frac{p_1(\bz)}{q_1(\bz)}\cdot\frac{p_2(\bz)}{q_2(\bz)},\]
where
\begin{equation}
  \label{pq}
\frac{p_1(\bz)}{q_1(\bz)}=\frac{
Q_\bipi(\bz)\,\prod_{m<l}(z_m-z_l) }{
\prod_{l=1}^k\prod_{\um\tau\leq
l}^{\tau\neq\pi_1\ldots \pi_l}(z_\tau-z_{\pi_{l}})}\text{ and
  }
\frac{p_2(\bz)}{q_2(\bz)} = \frac{
\alpha(\theta_1,\ldots, \theta_r, z_{\pi_1},\ldots ,z_{\pi_k})} {
\prod_{l=1}^k\prod_{i=1}^n(\lambda_i-z_l)}.
\end{equation}

Note that $p(\bz)$ is a polynomial, while $q(\bz)$ is a product of
linear forms.  As a first step we show that if $\pi_k \neq [k]$, then already the
first residue in the corresponding term on the right hand side of
\eqref{intnumberone} -- the one with respect to $z_k$ -- vanishes.
Indeed, if $\pi_k\neq[k]$, then $\deg(q_2(\bz);k)=n$, while
$z_k$ does not appear in $p_2(\bz)$. On the other hand, $\deg(q_1(\bz);k)=1$, because the only term which contains $z_k$ is the one corresponding to $l=k, \tau=[k]\neq \pi_k$. By \eqref{degree} $\deg(Q_\bipi(\bz),k)\le 1$ holds so 
\begin{equation}
\deg(p_1(\bz)p_2(\bz);k)=k \text{ and } \deg(q_1(\bz)q_2(\bz);k)=n+1
\end{equation}
and $k \le n-1$, so $\deg(p(\bz))\le \deg(q(\bz))+2$ holds and we can apply Lemma \ref{1lemma}.

We can thus assume that $\pi_k=[k]$, and proceed to the next step and take the residue 
with respect to $z_{k-1}$. If $\pi_{k-1}\neq[k-1]$ then 
\begin{equation}\label{k-1}
\deg(q_2(\bz),k-1)=\lead(q_2(\bz),k-1)=n, \deg(p_2(\bz);k-1)=0.
\end{equation}
In $q_1$ the linear terms containing $z_{k-1}$ are
\begin{equation}\label{termskminusone}
z_{k-1}-z_k , z_1+z_{k-1}-z_k, z_{k-1}-z_{\pi_{k-1}}
\end{equation}

The first term here cancels with the identical term in the Vandermonde in $p_1$. The second term divides $Q_{\bipi}$, according to the following proposition from \cite{bsz} applied for $l=k-1$:
\begin{proposition}[\cite{bsz}, Proposition 7.4]\label{divisible}
Let $l\geq1$, and let $\bipi$ be an admissible sequence of
  partitions of the form $\bipi=(\pi_1,\ldots ,\pi_l,[l+1],\ldots, [k])$, where $\pi_l\neq[l]$. Then
  for $m>l$, and every partition $\tau$ such that
  $l\in\tau$, $\um\tau\leq m$, and $|\tau|>1$, we have
  \begin{equation}
    \label{zdividesq}
 (z_\tau-z_m)|Q_\bipi.
  \end{equation}
\end{proposition}
Therefore, after cancellation, all linear factors from $q_1(\bz)$ which
have nonzero coefficients in front of both $z_{k-1}$ and $z_k$ vanish, and for the new fraction $\frac{p_1'(\bz)}{q_1'(\bz)}$
\[\deg(q'_1(\bz),k-1)=\lead(q'_1(\bz),k-1)=1.\]
By \eqref{termskminusone} and \eqref{degree} $\deg(Q_\pi,k-1)\le 3$ and therefore after cancellation we have 
\[\deg(p'_1(\bz),k-1)\le k-2+2=k\]
Using \eqref{k-1} we get  
\[\deg(p'_1(\bz)p_2(\bz),k-1)=k \text{ and } \deg(q'_1(\bz)q_2(\bz),k-1)=\lead(q'_1(\bz)q_2(\bz),k-1)=n+1,\]
so we can apply the second option in Proposition \ref{vanishprop} with $l=k-1$ to deduce the vanishing of the residue with respect to $k-1$.

In general, assume that 
\[\bipi=(\pi_1,\pi_2,\ldots ,\pi_l,[l+1],\ldots ,[k]) \text{ where } \pi_l\neq [l],\]
and proceed to the study of the residue with respect to $z_l$. For the second fraction we have again
\begin{equation}\label{ingeneral}
\deg(q_2(\bz),l)=\lead(q_2(\bz),l)=n, \deg(p_2(\bz);l)=0.
\end{equation}
The linear terms containing $z_l$ in $q_1(\bz)$ are 
\begin{eqnarray}
z_l-z_k,z_{l}-z_{k-1},\ldots ,z_{l}-z_{l+1} \label{vandweights} \\
z_\tau-z_s \text{ with } l \in \tau, \tau \neq l, l+1\le s\le k, \um \tau \le s    \label{restweights} \\
z_l-z_{\pi_l} \label{lastweight}
\end{eqnarray}
The weights in \eqref{vandweights} cancel out with the identical terms of the Vandermonde in $p_1(\bz)$ and by Propostition \ref{divisible} $Q_\bipi(\bz)$ is divisible by the weights in \eqref{restweights}. Hence all linear factors with nonzero coefficient in front of $z_l$ and at least one of $z_{l+1},\ldots, z_{k}$ vanish from $q_1(z)$. 
Let again $\frac{p_1'(\bz)}{q_1'(\bz)}$ denote the new fraction arising from $\frac{p_1(\bz)}{q_1(\bz)}$ after these cancellations. Then in $q'_1(\bz)$ only the term \eqref{lastweight} contains $z_l$ and  
\begin{equation}\label{q1general}
\deg(q'_1(\bz),l)=\lead(q'_1(\bz),l)=1.
\end{equation}
In $p'_1(\bz)$ the linear terms which are left from the Vandermonde after cancellation and contain $z_l$ are $z_{l-1}-z_l,\ldots, z_1-z_l$. The reduced $Q'_\pi(\bz)$ which we get after dividing by the terms in \eqref{restweights} is then a polynomial of the remaining weights, and the only remaining weights which contain $z_l$ are 
\[z_l-z_{\pi_l} \text{ and } z_l-z_k,z_l-z_{k-1},\ldots, z_l-z_{l+1}.\] 
Then \eqref{degree} tells us that $\deg(Q_\pi(\bz,l)\le k-l+1$. Therefore
\begin{equation}\label{p1general}
\deg(p'_1(\bz);l)\le (l-1)+(k-l+1)=k.
\end{equation} 
Putting \eqref{q1general} and \eqref{p1general} together we get
\[\deg(p'_1(\bz)p_2(\bz),l)=k \text{ and } \deg(q'_1(\bz)q_2(\bz),l)=\lead(q'_1(\bz)q_2(\bz),k-1)=n+1.\]
Since $k\le n-1$, by applying the second option of Proposition \ref{vanishprop} we arrive at the vanishing of the residue, forcing $\pi_l$ to be $[l]$.

\section{Increasing the number of points and the proof of Theorem \ref{main}}\label{sec:extendtherange}

The Residue Vanishing Theorem provides a closed iterated residue formula for tautological integrals on $\widetilde{CX}^{[k+1]}_p$ in the case when $k+1\le n$, that is, the number of points does not exceed the dimension of $X$. In this section we show how one can drop this very restrictive condition.

Recall that the test curve model in \S\ref{subsec:testcurve} establishes a $\GL(n)$-equivariant isomorphism of quasi-projective varieties  
\[\jetreg 1n/\jetreg 11 \simeq CX^{[k+1]}_p \subset \grass_k(\symdot)\]
between the moduli of $k$-jets of regular germs and the curvilinear locus of the punctual Hilbert scheme sitting in the Grassmannian of $k$-dimensional subspaces in $\symdot$. For punctual Hilbert schemes we can assume without loss of generality that $X=\CC^n$ and $p=0$ and we use the notation 
\[\CHilb^{k+1}_0(\CC^n) \subset \Hilb^{k+1}_0(\CC^n)\]
the curvilinear locus sitting in the punctual Hilbert scheme at the origin and $\overline{\CHilb}^{k+1}_0(\CC^n)$ for its closure, the curvilinear component.

Assume that $k+1>\dim(X)=n$. Fix a basis $\{e_1,\ldots, e_k\}$ of $\CC^k$ and let 
\[\CC_{[n]}=\mathrm{Span}(e_1,\ldots, e_n) \hookrightarrow \CC^{k+1} \text{ and } \CC_{[k+1-n]}=\mathrm{Span}(e_{n+1},\ldots, e_{k+1}) \hookrightarrow \CC^{k+1}\]
denote the subspaces spanned by the first $n$ and last $k+1-n$ basis vectors respectively. These are $T_{k+1}$-equivariant embeddings under the diagonal action of the maximal torus $T_{k+1}\subset \GL(k+1)$ and they induce the $T_{k+1}$-equivariant embedding
\begin{equation}\label{decomposition}
\jetreg 1n \hookrightarrow \jetreg 1{k+1}=\jetreg 1n \oplus \mathrm{Hom}(\CC^{k},\CC_{[k+1-n]}).
\end{equation}
defined via $f \mapsto (f,0)$. Here, by placing $f^{(i)}/i!\in \CC^n$ into the $i$th column we identify $\jetreg 1n$ with $\Hom^{\reg}(\CC^k,\CC^n)$, the set of $k$-by-$n$ matrices with nonzero first column and the decomposition \eqref{decomposition} reads as
\begin{multline}\label{decomposition2}
\jetreg 1{k+1}=\Hom^{\reg}(\CC^k,\CC^{k+1})=\Hom^{\reg}(\CC^k,\CC_{[n]})\oplus \Hom(\CC^k,\CC_{[k+1-n]})=\\
=\jetreg 1n \oplus \Hom(\CC^k,\CC_{[k+1-n]}).
\end{multline}

Moreover, $\jetreg 1n$ is invariant under the reparametrisation group $\jetreg 11$ acting on $\jetreg 1{k+1}$ and this action commutes with the $T_{k+1}$ action resulting a $T_{k+1}$-equivariant embedding 
\[\CHilb^{k+1}_0(\CC^n) \simeq \jetreg 1n/\jetreg 11 \subset \jetreg 1{k+1}/\jetreg 11=\CHilb^{k+1}_0(\CC^{k+1}).\]
This embedding extends to the closures and commutes with the embeddings into the Grassannians resulting the diagram 
\begin{equation*}
\xymatrix{\overline{\CHilb}^{k+1}_0(\CC^n) \ar[r]^-{T_{k+1}-equiv} \ar[d]^{\GL(n)-equiv} & \overline{\CHilb}^{k+1}_0(\CC^{k+1}) \ar[d]^{\GL(k+1)-equiv}\\
\grass_k(\symdot) \ar[r]^-{T_{k+1}-equiv} & \grass_k(\sym^{\le k}\CC^{k+1})} 
\end{equation*}
where the horizontal maps are $T_{k+1}$-equivariant and the vertical embeddings are $\GL(n)$ resp. $\GL(k+1)$-equivariant. 

%The decomposition of $f=(f',f'',\ldots, f^{(k)}) \in J_k(1,k)$ can be written in coordinates as  
%\[J_k(1,k)=J_k(1,n) \oplus \mathrm{Hom}(\CC^k,\CC_{[k-n]})\]
%\[f=(f',f'',\ldots, f^{(k)}) \mapsto (f_{[n]},f_{[k-n]})=((f'_{[n]},\ldots, f^{(k)}_{[n]}),(f'_{[k-n]},\ldots, f^{(k)}_{[k-n]}))\]
%where we use the shorthand notation $v_{[n]}=(v_1,\ldots, v_n)$ and $v_{[k-n]}=(v_{n+1},\ldots, v_k)$ for the projection of the vector $v\in \CC^k$ to $\CC_{[n]}$ and $\CC_{[k-n]}$ respectively. 
The decomposition \eqref{decomposition2} induces a $T_{k+1}$-equivariant isomorphism of quasi-projective varieties
\[\Psi^{k+1\to n}: \jetreg 1n/\jetreg 11 \times \mathrm{Hom}(\CC^{k},\CC_{[k+1-n]}) \overset{\simeq}\to \jetreg 1{k+1}/\jetreg 11 \]
\[(f_1 \cdot \jetreg 11,f_2) \mapsto (f_1 \oplus f_2)\cdot \jetreg 11\]
whose inverse on the open chart 
\[J_k^i(1,k+1)/\jetreg 11=\{(f',\ldots, f^{[k]})\cdot \jetreg 11:f'_i\neq 0\}\]
where the $i$th coordinate of $f'$ does not vanish can be given using a canonical slice of the action given by the following
\begin{lemma}
Let $n\ge 2$ and $f=(f',\ldots, f^{[k]}) \in J_k^i(1,k+1)$. The $\jetreg 11$-orbit of $f$ contains a unique point $\tilde{f}=(\tilde{f}',\ldots, \tilde{f}^{[k]})$ such that $\tilde{f}'_i=1$ and $\tilde{f}^{[j]}_i=0$ for $j\ge 2$.      
\end{lemma}

\begin{proof} 
$J_k^i(1,k+1)$ consists of $k+1$-by-$k$ matrices $(f',\ldots f^{(k)})$ whose $(i,1)$ entry $f'_i$ is nonzero $f'_i\neq 0$. The action of $\jetreg 11$ is right multiplication with the matrix given in \eqref{jetdiffmatrix}. This action multiplies the first column $f'$ with $\alpha_1$ whereas the image of $f^{(j)}$ for $2\le j \le k$ is
\[\alpha_j f'+\sum_{\substack{\tau \in \mathcal{P}(j)\\ |\tau|=s}}\alpha_\tau \cdot f^{(s)}.\] 
We choose the free parameter $\alpha_j$ inductively as $\alpha_1=1/f'_i$ and $\alpha_j=-\frac{1}{f'_i}\sum_{\substack{\tau \in \mathcal{P}(j)\\ |\tau|=s}}\alpha_\tau \cdot f^{(s)}_i$ for $2\le j \le k$
to get the desired form of the matrix. 
\end{proof}

The $T_{k+1}$-equivariant inverse of $\Psi^{k+1\to n}$ on the open chart $J_k^i(1,k+1)/\jetreg 11$ is then given as
\[\Psi_i^{n\to k+1}: J_k^i(1,k+1)/\jetreg 11 \to \jetreg 1n/\jetreg 11 \oplus \mathrm{Hom}(\CC^{k},\CC_{[k+1-n]})\]
\[(f_1 \oplus f_2)\cdot J_k(1,1) \mapsto (f_1 \cdot J_k(1,1),\tilde{f}_2).\]
This $T_{k+1}$-equivariant isomorphism gives us 
\begin{proposition}\label{ntok}
For any $T_{k+1}$-equivariantly closed compactly supported form $\mu$ on \\$\jetreg 1{k+1}/\jetreg 11$ we have
\[\int_{\jetreg 1n/\jetreg 11} \mu=\int_{\jetreg 1{k+1}/\jetreg 11}\mu \cdot \mathrm{Euler}^{T_k}(\mathrm{Hom}(\CC^{k},\CC_{[k+1-n]}))\]
where $\mathrm{Euler}^{T_k}(\mathrm{Hom}(\CC^{k},\CC_{[k+1-n]}))$ is the $T_{k+1}$-equivariant Euler class of $\mathrm{Hom}(\CC^{k},\CC_{[k+1-n]})$.
\end{proposition}

\begin{proof} 
We us the topological definition of equivariant duals, see Remark \ref{remark:topdef} and \cite{fultonnotes,kaz97,eg97} for details. 
We use the shorthand notations $\mathcal{J}_{k+1}=\jetreg 1{k+1}/\jetreg 11$ and $\mathcal{J}_n=\jetreg 1n/\jetreg 11$. The key observation is that $ET_{k+1} \times_{T_k} \mathcal{J}_n$ forms the zero section of the $T_{k+1}$-equivariant bundle $ET_{k+1} \times_{T_{k+1}} \mathcal{J}_{k+1}$ with fibres isomorphic to $\mathrm{Hom}(\CC^{k},\CC_{[k+1-n]})$. The (ordinary) Poincar\'e dual of the zero section is given by the top Chern class of the bundle, that is the $T_{k+1}$-equivariant Euler class of the fibre and therefore 
\[\int_{ET_{k+1} \times_{T_{k+1}} \mathcal{J}_n} \mu=\int_{ET_{k+1} \times_{T_{k+1}} \mathcal{J}_{k+1}}\mu \cdot c^{top}\]
for any compactly supported equivariantly closed $\mu \in H^*(ET_k \times_{T_k} \mathcal{J}_k)$. 
\end{proof}

As a corollary we get the following  
\begin{corollary}[\textbf{Extended Residue Vanishing Theorem}]
Formula \eqref{intnumberthree} remains valid for any $2\le n<k+1$.     
\end{corollary}

\begin{proof}
The weights of the $T_{k+1}$ action on $\mathrm{Hom}(\CC^{k},\CC_{[k+1-n]})$ in Proposition \ref{ntok} are the weights of the $T_{k+1}$ action on $f^{[i]}_j$ for $1\le i \le k$ and $n+1 \le j \le k+1$. The embedding $\phi^\flag:\jetreg 1{k+1}/\jetreg 11 \hookrightarrow \flag_k(\sym^{\le k}\CC^{k+1})$ is $T_{k+1}$-equivariant, and over the flag $\ff_\s$ the weight of $f^{[i]}_j$ is $\lambda_{\s(i)}-\lambda_{\s(j)}$. In the iterated residue formula of Proposition \ref{propflag} we write $\lambda_i-z_j$ for this weight and therefore the $T_{k+1}$-equivariant Euler class transforms into 
\[\mathrm{Euler}^{T_{k+1}}_{\bz}(\mathrm{Hom}(\CC^{k},\CC_{[k+1-n]}^{\bz}))=\prod_{i=1}^k \prod_{j=n+1}^k (\lambda_j-z_i)\]
over the flag $\ff_\s$ corresponding to an iterated pole $\bz=(z_1,\ldots z_k)$. If $\alpha=\alpha(\theta_1,\ldots, \theta_r,\eta_1,\ldots, \eta_k)$ is a bi-symmetric polynomial in the Chern roots $\theta_i$ of the pull-back of $F$ over $\widetilde{CX}^{[k+1]}_p=\CHilb^{k+1}_0(\CC^n) \subset \flag_k(\symdot)$ and the Chern roots $\eta_j$ of the tautological rank $k$ bundle $\cale$, then $\alpha$ is the restriction of a closed form on $\flag_k(\Sym^{\le k}\CC^{k+1})$ and in particular it is a restriction of a form on $\CHilb^{k+1}_0(\CC^{k+1})$. Therefore Remark \ref{remark:resolution}, Proposition \ref{ntok} and Theorem \ref{vanishtheorem} tell us that 
\begin{multline}\nonumber
\int_{\overline{CX}^{[k+1]}_p}\alpha=\sires \frac{
Q_k(\bz)\,\prod_{m<l}(z_m-z_l) \alpha(\theta,\bz) \,\dbz}{
\prod_{m+r \le l \le k}
(z_m+z_r-z_l)  \prod_{l=1}^k\prod_{i=1}^k(\lambda_i-z_l)}  \cdot \prod_{i=1}^k \prod_{j=n+1}^k (\lambda_j-z_i)=\\
=\sires \frac{
Q_k(\bz)\,\prod_{m<l}(z_m-z_l) \alpha(\theta,\bz)}{
\prod_{m+r \le l \le k}
(z_m+z_r-z_l)  \prod_{l=1}^k\prod_{i=1}^n(\lambda_i-z_l)} \,\dbz.
\end{multline}
\end{proof}

\subsection{Proof of Theorem \ref{main} and final remarks}\label{sec:final}
The weights $\lambda_1,\ldots, \lambda_n$ are the Chern roots of $T_p^*X$ and therefore $-\lambda_1,\ldots, -\lambda_n$ are the weights on $T_pX$. Theorem \ref{main} follows from the Residue Vanishing Theorem by substituting 
\[\frac{1}{\prod_{i=1}^n(\l_i-z_j)}=\frac{(-1)^n}{z_j^nc(1/z_j)}=(-1)^n\frac{s_X(1/z_j)}{z_j^n}\] 
If we give the $z_i$'s and $\theta_j$'s degree $1$ then the total degree of the rational expression 
\[\frac{(-1)^{nk}\prod_{i<j}(z_i-z_j)Q_k(z)M(c_i(z_i+\theta_j,\theta_j))}{\prod_{i+j\le l\le k}(z_i+z_j-z_l)(z_1\ldots z_k)^n}\]
in the formula is $n-k$.

The Chern class $c_i(z_i+\theta_j,\theta_j)$ is the coefficient of $t^i$ in 
\[c(F^{[k+1]})(t)=\prod_{j=1}^r(1+\theta_jt)\prod_{i=1}^k\prod_{j=1}^r(1+z_it+\theta_jt),\] 
that is, the $i$th Chern class of the bundle with formal Chern roots $\theta_j,z_i+\theta_j$. For example
\[c_1(z_i+\theta_j,\theta_j)=(k+1)\sum_{j=1}^r \theta_j+r\sum_{i=1}^k z_i,\] 
and in general $c_i(z_i+\theta_j,\theta_j)$ is a degree $i$ polynomial of the form 
\[c_i(z_i+\theta_j,\theta_j)=A_ic_{i}(\bz)+A_{i-1}c_{i-1}(\bz)+\ldots +A_0\]
where $c_j(\bz)$ is the $j$th elementary symmetric polynomial in $z_1,\ldots, z_k$ and $A_j$ is a degree
$n-j$ symmetric polynomial in $\theta_1,\ldots, \theta_r$.

\end{document}